\documentclass[a4paper,11pt]{amsart}
\usepackage[left=2.5cm, right=2.5cm,top=2.5cm,bottom=3cm]{geometry}

\usepackage[utf8]{inputenc}
\usepackage[english]{babel}
\usepackage{amssymb}
\usepackage{amsfonts}
\usepackage{dsfont}

\usepackage{hyperref}

\usepackage{tikz}
\usepackage{chemarr}

\usepackage{graphicx}
\graphicspath{ {./} }
\usepackage{hhline}
\usepackage[sort&compress,comma,square,numbers]{natbib}

\usepackage{tikz-cd}
\usepackage{algorithm}
\usepackage{algorithmic}

\DeclareMathOperator{\diag}{diag}

\DeclareMathOperator{\Log}{Log}
\DeclareMathOperator{\Exp}{Exp}

\DeclareMathOperator{\im}{im}

\DeclareMathOperator{\inte}{int}
\DeclareMathOperator{\relint}{relint}
\DeclareMathOperator{\Conv}{Conv}
\DeclareMathOperator{\conv}{Conv}
\DeclareMathOperator{\Cone}{Cone}

\DeclareMathOperator{\Trop}{Trop}
\DeclareMathOperator{\Sing}{Sing}

\DeclareMathOperator{\rk}{rk}
\DeclareMathOperator{\pr}{pr}
\DeclareMathOperator{\sign}{sign}
\DeclareMathOperator{\Hess}{Hess}
\newcommand{\R}{\mathbb{R}}
\newcommand{\Rn}{\R^n}

\newtheorem{thm}{Theorem}[section]
\newtheorem{cor}[thm]{Corollary}
\newtheorem{prop}[thm]{Proposition}
\newtheorem{lemma}[thm]{Lemma}

\newtheorem{conj}[thm]{Conjecture}

\theoremstyle{definition}

\newtheorem{ex}[thm]{Example}
\newtheorem{remark}[thm]{Remark}

\newtheorem{assumption}[thm]{Assumption}

\newcommand{\cA}{\mathcal{A}}
\newcommand{\hA}{\hat{A}}
\newcommand{\eps}{\varepsilon}

\begin{document}

\title{Viro's patchworking and the signed reduced A-discriminant}

\author{Weixun Deng}
\address{Department of Mathematics, Texas A\&M University,
College Station, Texas, USA}
\email{deng15521037237@tamu.edu}

\author{J. Maurice Rojas}
\address{Department of Mathematics, Texas A\&M University,
College Station, Texas, USA}
\email{rojas@tamu.edu}

\author{Máté L. Telek}
\address{Department of Mathematical Sciences, University of Copenhagen,
Universitetsparken 5,
2100 Copenhagen, Denmark}
\email{mlt@math.ku.dk}

\begin{abstract}
Computing the isotopy type of a hypersurface, defined as the positive real zero set of a multivariate polynomial, is a challenging problem in real algebraic geometry. We focus on the case where the defining polynomial has combinatorially restricted exponent vectors and fixed coefficient signs, enabling faster computation of the isotopy type. In particular, Viro’s patchworking provides a polyhedral complex that has the same isotopy type as the hypersurface, for certain choices of the coefficients. So we present
properties of the signed support, focussing mainly on the case of n-variate (n+3)-nomials, that ensure all possible isotopy types can be obtained via patchworking. To prove this, we study the signed reduced A-discriminant and show that it has a simple structure if the signed support satisfies some combinatorial conditions.

\medskip
\emph{Keywords:  exponential sum, signed support, Viro's patchworking, isotopy type}
\end{abstract}

\maketitle


\section{Introduction}
Tropical geometry bridges the worlds of algebraic and polyhedral geometry. The 
key idea is to transform algebraic varieties into polyhedral objects, turning 
their algebraic structure functorially into a combinatorial structure. This 
approach has been fruitful in the case of varieties over algebraically closed fields \cite{maclagan2015introduction,TropAlgGeo_book}. Recently, the tropicalization of semialgebraic sets over the real numbers (or more generally over real closed fields) has received increasing attention \cite{SpeyerWilliams::PosGrassmannian,Vinzant::RealRadical,JellScheidererYu::RealTrop,Blekherman::Moments,Brandenburg::TropicalPositivity}. 

In the early  1980s, Oleg Viro showed that it is possible to associate a polyhedral complex with a parametrized polynomial in such a way that there exists \textit{some} choice of coefficients such that the polyhedral complex and the positive real zero set of the polynomial with that choice of coefficients have the same isotopy type \cite{Viro_Dissertation}. This result is known as  \emph{Viro's patchworking} in the literature and was one of the first examples of tropical geometry.

Classifying the isotopy type of real hypersurfaces is a challenging question in real algebraic geometry, tracing its origins to Hilbert's 16th problem  \cite{Viro::Hilbert16}. In his original formulation, Hilbert asked for a classification of isotopy types of plane real algebraic curves. Based on his patchworking method, Viro provided such a classification for curves of degree $7$ \cite{Viro::Degree7}.

Using Viro's method, under certain conditions, one can determine the possible 
isotopy types for hypersurfaces defined by the positive roots of 
polynomials of the form $f(x)\!=\!\sum_{\alpha\in \cA} c_\alpha x^\alpha$, with $\cA \subseteq \mathbb{Z}^n$ a finite set and $c_\alpha$ 
real and nonzero for all $\alpha\!\in\!\cA$. We call $\cA$ the {\em support} of 
$f$, let $\varepsilon_\alpha:=\sign(c_\alpha)$, $\eps\!:=\!(\eps_\alpha\; | \; \alpha\!\in\!\cA)\!\in\!\{\pm 1\}^{\cA}$ denote 
the vector of signs of the 
coefficients,
and call the pair $(\cA,\eps)$ the {\em signed 
support} of $f$. To know whether one can build {\em all} possible isotopy 
types for a given $(\cA,\eps)$ via Viro's patchworking one must first 
understand how the isotopy type of a hypersurface changes while 
varying the coefficient vector $c$. We will abuse notation slightly by 
calling $\eps_\alpha$ the {\em sign} of the exponent vector $\alpha$ when the underlying 
polynomial and support are clear. Hence we will sometimes speak of 
{\em positive} or {\em negative} exponents in this sense.

It is known that using discriminant varieties, one can decompose the coefficient space into a disjoint union of open connected sets --- called {\em chambers} 
---  such that in each chamber the isotopy type of any corresponding 
hypersurface is constant \cite{gelfand1994discriminants, BihanBound}. More specifically, such a chamber decomposition is given by connected components of the complement of the union of the signed $A$-discriminants $\nabla_{A_{F},\eps_F}$ for the faces $F$ of $\Conv(\cA)$. In \cite{RojasRusek_Adiscriminant}, a 
reduced version of the signed $A$-discriminant, denoted 
$\Gamma_\eps(A,B)$,  was introduced. The complement of the union of the 
signed reduced $A$-discriminants has the same property as its non-reduced 
counterpart, but has the advantage of reducing the number of parameters. 
We will review these constructions in Section \ref{Sec::SignedDiscr}. 

When considering chambers (reduced or non-reduced), there will sometimes 
be chambers where the isotopy type of the hypersurfaces intersected with the positive real orthant can {\em not} be 
obtained by Viro's patchworking (see, e.g., 
\cite[Example 2.9]{forsgard2017new} and \cite[Theorem 7.8]{BihanBound}). 
We call these chambers \emph{inner chambers}, see Remark~\ref{Remark:InnerChambers} for an explanation of this naming convention.

Our main goal is to give conditions on the signed support $(\cA,\eps)$ 
enabling Viro's patchworking to find all isotopy types, i.e., conditions 
that obstruct the existence of inner chambers. For instance, we show that the 
signed $A$-discriminant is empty if and only if the exponent vectors in $\cA$ 
with positive signs can be separated from those with negative signs by an 
affine hyperplane (Theorem \ref{Thm::SepHyperplane}). In that case, all 
hypersurfaces with signed support $(\cA,\eps)$ have the same isotopy 
type, and the isotopy type can be obtained by Viro's patchworking. Moreover, 
for a given set of exponent vectors $\cA$, we give an upper bound on the number of sign distributions $\eps \in \{ \pm 1 \}^{\cA}$ such that the signed support $(\cA,\eps)$ does not have a separating hyperplane (Proposition \ref{Prop_NumOfNonTrivSepHyp}).

As we will see, it will be convenient (and natural) to assume $\conv (\cA)$ 
has dimension $n$ and cardinality $n+k+1$, and then consider $k$ as a 
measure of the complexity of the resulting family of isotopy types. 
The special cases $k\!\in\!\{0,1\}$ are addressed in \cite{BCDPRR,epr}, so we 
will contribute to the case $k\!=\!2$: 
We show that the complement of the signed reduced $A$-discriminant has exactly 
two connected components if $\cA$ contains exactly one negative exponent 
vector (Theorem  \ref{Thm::OneNegative}) or if the positive and the negative exponent vectors of $\cA$ are separated by an $n$-simplex in a certain way (Theorem \ref{Thm::Simplex}). Under the additional assumption that the signed $A$-discriminants associated to proper faces of $\Conv(\cA)$ are empty, one can use Viro's patchworking to find all possible isotopy types (Corollary \ref{Thm::ViroAll}). 

Signed supports with a separating hyperplane or a separating simplex have previously been studied in \cite{DescartesHypPlane}. In that work, the authors used these conditions to show that the set of points where the polynomial takes negative values has, at most, one connected component.

For $n=2$ and $\cA$ consisting $5$ of points, we show that if the negative and positive exponent vectors are separated by two pairs of affine lines (Theorem \ref{Thm_TwoEnclHyp}), then  the complement of $\Gamma_\eps(A,B)$ has at most two connected components, and each is unbounded. In Section \ref{Section::5nomials}, we study further bivariate $5$-nomials and show that for a bounded chamber to exist in the complement of $\Gamma_\eps(A,B)$, the set of exponent vectors must satisfy very restrictive inequalities, see Theorem \ref{Thm::5nomials} and Remark \ref{Remark::5nomials}. 

We will work in the more general context of exponential sums on 
$\R^n$, instead of polynomials on $\R^n_{>0}$, and this will 
in fact simplify some of our arguments. Note that any real polynomial can be 
transformed into a real exponential sum while preserving topological properties of the corresponding zero sets: Any polynomial $f\colon \R^{n}_{>0} \to \R$ give rise to an exponential sum $\R^n \to \R$, $(x_1, \dots x_n) \mapsto f(e^{x_1}, \dots, e^{x_n} )$. Since the map $\Exp\colon \R^n \to \R^n_{>0},\, (x_1, \dots x_n) \mapsto (e^{x_1}, \dots, e^{x_n} )$ is a homeomorphism, two subsets of $\R^n_{>0}$ have the same isotopy type if and only if their images under $\Exp$ have the same isotopy type.

 \subsection*{Notation} For two vectors $v,w \in \R^{n}$, $v \cdot w = v_1w_1 + \dots + v_n w_n$ denotes the Euclidean scalar product, and $v \ast w = (v_1w_1, \dots ,v_n w_n)$ denotes the coordinate-wise product of $v$ and $w$. The transpose of a matrix $M$ will be denoted by $M^\top$. We denote the interior of a set $X \subseteq \R^{n}$ by $\inte(X)$. If $X \subseteq \R^{n}$ is a polyhedron, $\relint(X)$ denotes the relative interior of $X$. By $\#S$ we denote the cardinality of a finite set $S$. For a differentiable function $f \colon \R^n \to \R^m$, the Jacobian matrix at $x \in \R^n$ is denoted by $J_f(x)$.

\section{Preliminaries}

\subsection{Signed support of an exponential sum}
Let $\cA = \{\alpha_1, \dots , \alpha_{n+k+1} \} \subseteq \R^{n}$ be a finite set. We think about the elements of $\cA$ as the exponent vectors of an \emph{exponential sum}:
\begin{align}
\label{Eq::ExpSum}
f_c\colon \R^{n} \to \R, \quad x \mapsto f_c(x) = \sum_{i=1}^{n+k+1} c_i e^{\alpha_i\cdot x},
\end{align}
where $c \in (\R \setminus \{0\})^{n+k+1}$. We call $\eps = \sign( c) \in \{ \pm 1 \}^{n+k+1} $ a \emph{sign distribution}, and $(\cA,\eps)$ a \emph{signed support}. For a fixed order of the exponent vectors $\alpha_1, \dots , \alpha_{n+k+1}$, there is an isomorphism between the vector spaces $\R^{\cA} \cong \R^{n+k+1}$. We might use these two notations interchangebly. For a fixed sign distribution $\eps \in \{ \pm 1\}^{n+k+1}$, we write 
\[ \R^{\cA}_\eps = \{ c \in \R^{\cA} \mid \sign(c) = \eps \}.\]
for the orthant in $\R^\cA$ containing the coefficients matching the signs given by $\eps$.

We split the support set $\cA$ into the sets of \emph{positive} and \emph{negative exponent vectors}, that is, we define
\[ \cA_+ := \{ \alpha_i \in \cA \mid \eps_i = 1 \}, \qquad  \cA_- := \{ \alpha_i \in \cA \mid \eps_i = -1 \}.\]
We call $\cA$ \emph{full-dimensional} if $\Conv(\cA)$ has dimension $n$. For a set $S \subseteq \R^{n}$, we denote the restriction of $f_c$ to $S$ by
\[ f_{c|S}\colon \R^{n} \to \R, \quad x \mapsto f_c(x) = \sum_{ \alpha_i \in \cA \cap S } c_i e^{\alpha_i\cdot x}.\]
 Furthermore, we set $\mathcal{A}_S:=\mathcal{A}\cap S$ and
   $\varepsilon_S:=(\sign(c_i) \; | \; \alpha_i \in \cA_S)$.

The real zero set of $f_c$ is denoted by
\[ Z(f_c) := \{ x \in \R^n \mid f_c(x) = 0 \}. \]
We are interested in the possible \emph{isotopy types} of $Z(f_c)$ when $c \in \R^{n+k+1}$ varies over all coefficients such that $\sign(c) = \eps$. Two subsets $Z_0, Z_1 \subseteq \R^n$ are \emph{isotopic} (ambient in $\R^n$) if there exists a continuous map $H\colon [0,1] \times \R^n \to \R^n$, called an \emph{isotopy}, such that
\begin{itemize}
\item $H(t, \cdot)$ is a homeomorphism for all $t \in [0,1]$,
\item $H(0,\cdot)$ is the identity on  $\R^n$,
\item $H(1,Z_0) = Z_1$.
\end{itemize}
Being isotopic gives an equivalence relation on subsets of $\R^n$ \cite[Chapter 10.1]{Book::Isotopy}, which allows us to talk about their isotopy types.

\subsection{Viro's patchworking}

Viro's patchworking method provides possible isotopy types of $Z(f_c), \, c \in \R_\eps^\cA$ for a fixed signed support $(\cA,\eps)$.  In this section, we recall this method. We follow the notation used in \cite{RealHomotopy}.

Let $\cA = \{ \alpha_1, \dots , \alpha_{n+k+1} \}   \subseteq \mathbb{Z}^n$ be a finite set and $h \in \R^{n+k+1}$. We consider the \emph{lifted} points
\[ \cA^h := \big\{ (\alpha_i,h_i) \in \R^{n+1} \mid \alpha_i \in \cA \big\}.\]
A face $F \subseteq \Conv(\cA^h)$ is called an \emph{upper face} if there exists a vector $(v_F,1) \in \R^{n+1}$ such that
\[ F = \big\{ x \in \Conv(\cA^h) \mid (v_F,1) \cdot x \geq (v_F,1) \cdot y \text{   for all } y \in \Conv(\cA^h)\big\}. \]
The projection of upper faces of $\Conv(\cA^h)$ onto $\R^n$ gives a polyhedral subdivision $\mathcal{P}$ of $\Conv(\cA)$. For a generic choice of  $h \in \R^{n+k+1}$, each upper facet $F \subseteq \Conv(\cA^h)$ contains exactly $n+1$ points of $\cA^h$, and each polyhedron in $\mathcal{P}$ is a simplex.

The \emph{tropical hypersurface} associated to $\cA$ and $h \in \R^\cA$ is defined as
\[ \Trop(\cA,h):= \big\{ v \in \R^n \mid \max_{i = 1, \dots , n+k+1}( v \cdot \alpha_i + h_i) \text{ is attained at least twice}.\big\}.\] 
It is dual to the $(n-1)$-skeleton of the subdivision $\mathcal{P}$  induced by $h$. For a sign distribution $\eps \in \{\pm1 \}^\cA$, we define the \emph{signed tropical hypersurface}
\[ \Trop_\eps(\cA,h) := \Big\{ v \in \R^n \mid \max_{i = 1, \dots , n+k+1}( v \cdot \alpha_i + h_i) \; \text{ is attained for some } \alpha \in  \cA_+ \text{ and } \alpha' \in \cA_- \Big\}. \] 

\begin{ex}
\label{Example_Viro}
  Consider the following set of exponent vectors
    \begin{align}
 \cA = \{\alpha_1,\,\alpha_2,\,\alpha_3,\,\alpha_4,\,\alpha_5 \} = \left\{
  \begin{bmatrix} 0\\ 0 \end{bmatrix},  \,  \begin{bmatrix} 1 \\ 0 \end{bmatrix},   \,    \begin{bmatrix} 0 \\ 1 \end{bmatrix} ,   \,  \begin{bmatrix} 4 \\ 1 \end{bmatrix} ,   \,   \begin{bmatrix} 1 \\ 4 \end{bmatrix} \right\},
  \end{align}
   the sign distribution $\eps = (1,1,1,-1,-1)$ and $h = (5,4,4,5,5)$. The upper faces of the convex hull of the lifted points 
      \begin{align*}
 \cA^h = \left\{ \begin{bmatrix} 5 \\0\\ 0 \end{bmatrix},  \quad   \begin{bmatrix} 4\\ 1 \\ 0 \end{bmatrix},   \quad    \begin{bmatrix} 4\\ 0 \\ 1 \end{bmatrix} ,   \quad    \begin{bmatrix}5\\ 4 \\ 1 \end{bmatrix} ,   \quad    \begin{bmatrix}5\\ 1 \\ 4 \end{bmatrix} \right\},
  \end{align*}
  are shown in Figure \ref{FIG_Example_Viro}(a). The induced subdivision of $\Conv(\cA)$ contains $3$ triangles, $7$ edges, and $5$ vertices. The tropical hypersurface $\Trop(\cA,h)$, which is dual to the $1$ skeleton of the subdivision, contains $3$ vertices and $7$ edges. We depicted $\Trop(\cA,h)$ and the signed tropical hypersurface $\Trop_\eps(\cA,h)$ in Figure \ref{FIG_Example_Viro}(b). 
\end{ex}

\begin{figure}[t]
\centering
\begin{minipage}[h]{0.45\textwidth}
\centering
\includegraphics[scale=0.45]{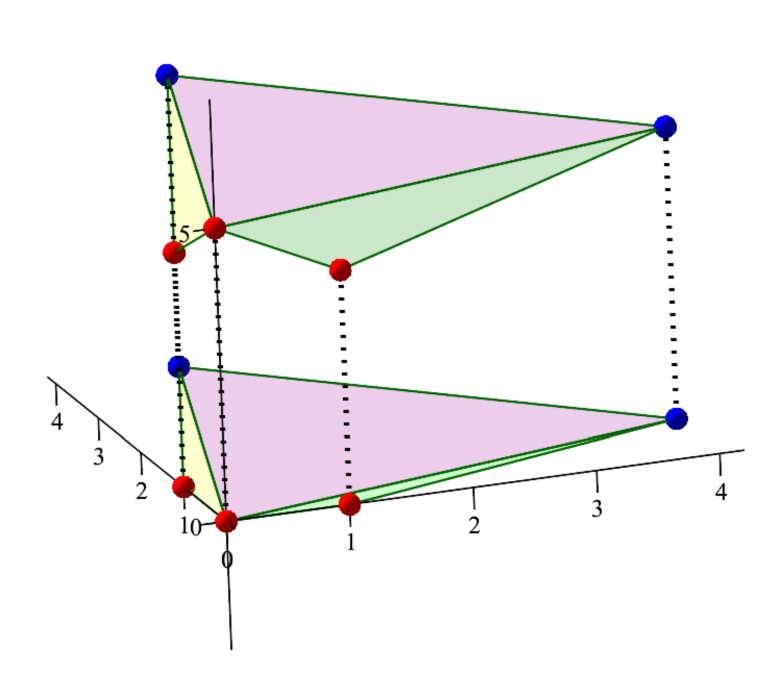}

{\small (a)}
\end{minipage}
\begin{minipage}[h]{0.45\textwidth}
\centering
\includegraphics[scale=0.5]{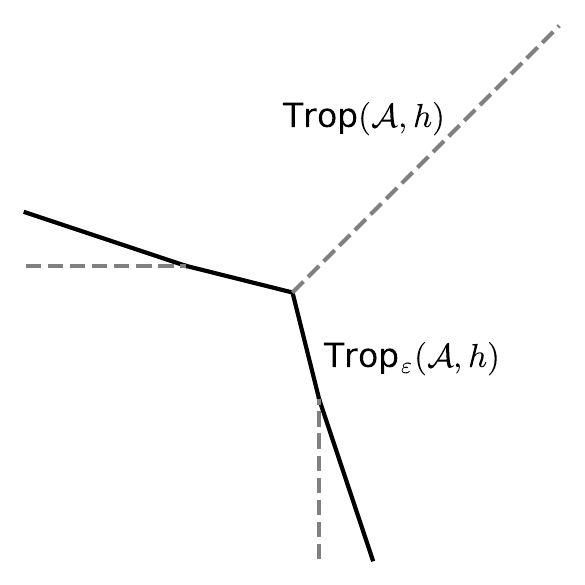}

{\small (b)}
\end{minipage}
\caption{{\small (a) Upper convex hull of the lifted points from Example \ref{Example_Viro} and the induced polyhedral subdivision. (b)  Signed tropical hypersurface associated to the points in (a). }}\label{FIG_Example_Viro}
\end{figure}

The following theorem is a well-known version of Viro's patchworking. Although the cited sources refer to polynomials and their zero sets in the positive orthant $\R_{>0}^n$, the result immediately extends to exponential sums using the coordinate-wise exponential map $\Exp\colon \R^n \to \R_{>0}^n$.

\begin{thm}
\label{Thm:Viro}
\cite{Viro_Dissertation}\cite[Ch.11 Theorem 5.6]{gelfand1994discriminants}\cite[Theorem 2.19]{TropAlgGeo_book}
Let $(\cA,\eps)$ be a signed support such that $\cA \subseteq \mathbb{Z}^n$ and let $h \in \R^\cA$ be generic.
For $t  \in \R$, consider the exponential sum
\[ g_t  \colon \R^{n} \to \R, \quad x \mapsto \sum_{i = 1}^{n+k+1} \eps_i e^{h_i t} e^{\alpha_i \cdot x}.\]

Then the signed tropical hypersurface $\Trop_\eps(\cA,h)$ is isotopic to $Z(g_t)$ for $t \gg 0$ sufficiently large.
\end{thm}

\subsection{Signed $A$-discriminant}
\label{Sec::SignedDiscr}
The goal of this subsection is to recall the notion of the \emph{A-discriminant} from \cite{gelfand1994discriminants, RojasRusek_Adiscriminant, BihanBound}. Let $f_c$ be an exponential sum as in \eqref{Eq::ExpSum}. A point $x \in \R^n$ is a \emph{singular zero} of $f_c$ if and only if 
\begin{align}
\label{Eq::EqsForSingularZeros}
 f_c(x) = \frac{\partial f_c(x) }{\partial x_1} = \dots =  \frac{\partial f_c(x) }{\partial x_n} = 0.
 \end{align}
We denote the set of singular zeros of $f_c$ by $\Sing(f_c)$. For a fixed signed support $(\cA,\eps)$, we define the \emph{signed $A$-discriminant} as 
\begin{align*}
    \nabla_{\cA,\eps} := \big\{ c \in \R^{\cA}_\eps \mid \Sing(f_c) \neq \emptyset \big\}.
\end{align*}
Thus, $\nabla_{\cA,\eps}$ contains all coefficients $c \in \R^{\cA}_\eps$ such 
that the exponential sum $f_c$ has a singular zero in $\Rn$. 

For the sake of completeness, we recall that the signed $A$-discriminant does not change under affine transformations of the exponent vectors.
\begin{prop}
\label{Lemma::Transform}
Let $f_c$ be an exponential sum with support $\cA = \{ \alpha_1, \dots , \alpha_{n+k+1} \} \subseteq \R^n$. For an invertible matrix $M \in \R^{n \times n}$ and $v \in \R^n$ consider the exponential sum
\[ g_c \colon \R^n \to \R, \quad x \mapsto g_c(x) = \sum_{i = 1}^{n+k+1} c_i e^{ (M \alpha_i  + v)  \cdot x}.\]
Then we have:
\begin{itemize}
    \item[(i)]  If $\det(M) > 0$, then the hypersurfaces $Z(f_c)$ and $Z(g_c)$ are isotopic.
    \item[(ii)] $\Sing(f_c) = M^\top\Sing(g_c)$.
    \item[(iii)] For all $x \in  \Sing(g_c)$ the Hessian matrices $\Hess_{f_c}(M^\top x)$ and $\Hess_{g_c}(x) $ have the same number of positive, negative and zero eigenvalues.
\end{itemize}
\end{prop}

\begin{proof}
Note that $g_c(x) = e^{v\cdot x}f_c(M^\top x)$ for all $x \in \R^n$. Since $e^{v\cdot x} \neq 0$, we have 
\begin{equation}
\label{Lemma_AffineTrans}
\begin{aligned}
    Z(g_c) &= Z(f_c(M^\top x))  =\left\{x \in \R^n \mid \sum\limits_{i = 1}^{n+k+1} c_i e^{ \alpha_i \cdot (M^\top x)}=0\right\} \\
   &=\left\{(M^\top)^{-1}y \in \R^n \mid \sum\limits_{i = 1}^{n+k+1} c_i e^{ \alpha_i \cdot y}=0\right\} = (M^\top)^{-1} Z(f_c)
\end{aligned}
\end{equation}
 Since the group of invertible real $n\times n$ matrices with positive determinant is path-connected (see, e.g. \cite[Theorem 3.68]{DiffManifoldandLieGroup_book}), there exists a continuous path from the identity matrix to $(M^\top)^{-1}$, which induces an isotopy. This shows (i).

Applying the product and the chain rule from calculus, we have
\[ J_{g_c}(x) = v^\top f_c(M^\top x) + e^{v\cdot x} J_{f_c}(M^\top x)M^\top.\]
Using \eqref{Lemma_AffineTrans} and that $M^\top$ is invertible, for $x \in Z(g_c)$  it follows that  $J_{g_c}(x) = 0$ if and only if  $J_{f_c}(M^\top x) = 0$, which implies (ii).

For the rest of the proof, we assume that $x \in \Sing(g_c)$. From \cite[Corollary 1]{skorski2019chain} it follows that
\[ \Hess_{g_c}(x) = e^{v\cdot x} M \Hess_{f_c} (M^\top x) M^\top. \]
Thus, $\Hess_{g_c}(x)$ and $\Hess_{f_c}(M^\top x)$ have the same number of positive, negative and zero eigenvalues by Sylvester's law of inertia \cite[Chapter 7]{meyer04}.
\end{proof}

\begin{cor}
\label{Cor:AdiscNotChanged}
    Let $(\cA,\eps)$ be a signed support, $M \in \R^{n\times n}$ an invertible matrix and $v \in \R^n$. For $M \cA+ v = \{ M\alpha+v \mid \alpha \in \cA \}$, we have
    \[ \nabla_{\cA,\eps}  = \nabla_{M \cA+v,\eps}.\]
\end{cor}
\begin{proof}
    The statement follows directly from Proposition \ref{Lemma::Transform}(ii).
\end{proof}

\begin{remark}
Using Proposition \ref{Lemma::Transform}, one might transform any full-dimensional support $\cA \subseteq \R^n$ to a support containing the standard basis vectors $e_1, \dots, e_n \in \R^n$ and the zero vector without changing the isotopy types of the corresponding hypersurfaces.

To be more precise, from $\dim \Conv(\cA) = n$ it follows that $\cA$ contains $n+1$ affinely independent vectors $\alpha_1, \dots , \alpha_{n+1}$. Thus, there exists an invertible matrix $M \in \R^{n\times n}$ such that $M(\alpha_1 -\alpha_{n+1}) = e_1, \dots, M(\alpha_n-\alpha_{n+1}) = e_n$. If $\det(M) < 0$, we change the order of $\alpha_1, \dots, \alpha_n$ such that the corresponding matrix $M$ has positive determinant. For $v:=-M\alpha_{n+1}$, the affine linear map $L\colon \R^n \to \R^n, \, \alpha \mapsto M\alpha + v$ satisfies $\{ 0,e_1, \dots ,e_n\} \subseteq L(\cA)$.
\end{remark}

For each face $F \subseteq \Conv(\cA)$, we define $\nabla_{\cA_F,\eps_F}$ in a similar way, and set
\[ \tilde{\nabla}_{\cA_F,\eps_F} := \big\{ (c_{\alpha_i} )_{i=1, \dots, n+k+1} \in \R^{n+k+1}_\eps \mid  (c_{\alpha_i})_{\alpha_i \in \cA_F } \in \nabla_{\cA_F,\eps_F} \big\}.\]

In \cite{BihanBound}, the authors proved the following statement regarding the topology of hypersurfaces corresponding to different connected components of the complement of the union of the signed $A$-discriminants $\tilde{\nabla}_{\cA_F,\eps_F}$.

\begin{prop}
\cite[Proposition 2.10]{BihanBound}
\label{Prop::BihanChambers}
Let $(\cA,\eps)$ be a full-dimensional signed support. If $c$ and $c'$ are in the same connected component of  
\[  \R^{\cA}_\eps \setminus \left( \bigcup_{F \subseteq \Conv( \cA )  \text{ a face}} \tilde{\nabla}_{\cA_F,\eps_F} \right),\]
then the zero sets $Z(f_c)$ and $Z(f_{c'})$ are isotopic.
\end{prop}

Finding the defining equalities of the signed $A$-discriminant is challenging, 
but an explicit parametrization is much simpler to find. First, let 
$\diag(c)$ denote the $\#\cA \times \#\cA$ diagonal matrix with 
$(a,a)$-entry $c_a$, and let us rewrite the equalities in 
\eqref{Eq::EqsForSingularZeros} as:
\begin{align}
\label{Eq_FullCriticalSystem}
 \hA \diag(c) \Big( e^{\alpha_i \cdot x} \Big)_{i=1,\dots,n+k+1} =  \hA_\eps \diag(\lvert c \rvert) \Big( e^{\alpha_i \cdot x} \Big)_{i=1,\dots,n+k+1} =0,
\end{align}
where the matrix $\hA$ is given by
\begin{align}
\label{Eq::AMatrix}
\hA=\left[\begin{array}{ccc}
1  & \ldots & 1 \\
\alpha_1 & \ldots & \alpha_{n+k+1}
\end{array}\right] \in \R^{(n+1) \times(n+k+1)},
\end{align}
$\eps = \sign( c) \in \{ \pm 1\}^{n+k+1}$ and $\hA_\eps = \hA \diag(\eps)$. We refer to the equation system \eqref{Eq_FullCriticalSystem} as the \emph{critical system} of $(\cA,\eps)$. Note that the assumption $\dim \Conv( \cA) = n$ is equivalent to $\rk(\hA) = n+1$. If $\rk(\hA) = n+1$, the kernel of $\hA$ has dimension $k$, which is usually called the \emph{codimension} of 
$\cA$.

If  $x \in \R^n$ is a singular zero of $f_c$, then $\diag(c) \big( e^{\alpha_i \cdot x} \big)_{i=1,\dots,n+k+1} \in \ker(\hA)$. Choose a basis of $\ker(\hA)$ and write these vectors as columns of a matrix $B \in \R^{(n+k+1) \times k}$. Such a choice of $B$ is called a \emph{Gale dual matrix} of $\hA$. With slight abuse of notation, we might call $B$ a Gale dual matrix of $\cA$. Since $\im(B) = \ker(\hA)$ for each $x \in \Sing(f_c)$ there exists $\lambda \in \R^k$ such that
\begin{align}
   \label{Eq::Blambda}
     B\lambda =  \diag(c) \big( e^{\alpha_i \cdot x} \big)_{i=1,\dots,n+k+1}.
\end{align}
Since $e^{\alpha_i \cdot x}$ is positive for all $x \in \R^n$ and all $\alpha_1, \dots , \alpha_{n+k+1}$, the signs of the vector in \eqref{Eq::Blambda} are given by $\sign(c) = \eps$. Therefore, it is enough to look at the set
\[ \mathcal{C}_{B,\eps} := \Big\{ \lambda \in \R^k  \mid \sign(B\lambda) = \eps \Big\}. \]
We define the  \emph{signed Horn-Kapranov Uniformization} map as
\begin{align}
\label{Eq::HornKapranov}
 \psi\colon  \mathcal{C}_{B,\eps} \times \R^n \to \R^{n+k+1}_\eps, \quad (\lambda,x) \mapsto B\lambda \ast \big( e^{\alpha_i \cdot (-x)} \big)_{i=1,\dots,n+k+1},
\end{align}
where $\ast$ denotes the point-wise multiplication of two vectors.

\begin{prop}
\label{Prop_HornKapranov}
Let $(\cA,\eps)$ be a full-dimensional signed support with Gale dual matrix $B$. For the  signed Horn-Kapranov Uniformization map \eqref{Eq::HornKapranov}, we have
\[ \im(\psi) = \nabla_{\cA,\eps}.\]
\end{prop}

\begin{proof}
If $\psi(\lambda,x) = c$, then
\[\hA \diag(c) \big( e^{\alpha_i \cdot x} \big)_{i=1,\dots,n+k+1} = \hA \Big(  B\lambda \ast \big( e^{\alpha_i \cdot -x} \big)_{i=1,\dots,n+k+1} \ast \big( e^{\alpha_i \cdot x} \big)_{i=1,\dots,n+k+1} \Big) = \hA B \lambda = 0, \]
which implies that $x \in \Sing(f_c)$ and therefore $c \in \nabla_{\cA,\eps}$.

On the contrary, if $c \in \nabla_{\cA,\eps}$, then there exists a point $x \in \Sing(f_c)$. From \eqref{Eq::Blambda} follows that there exists $\lambda \in \mathcal{C}_{B,\eps}$ such that $\psi(\lambda,x) = c$.
  \end{proof}

The signed $A$-discriminant lives in an ambient space of dimension $n+k+1$. Following \cite{RojasRusek_Adiscriminant}, we reduce the dimension of the ambient space to $k$ by quotienting out some homogeneities without losing essential information as follows.  We define the  \emph{signed reduced A-discriminant} $\Gamma_\eps(A,B)$ \cite[Definition 2.5]{RojasRusek_Adiscriminant} to be 
\[ \Gamma_\eps(A,B) := B^\top\Log\lvert \nabla_{\cA,\eps} \rvert,\]
where $\Log$ is the coordinate-wise natural logarithm map and $\lvert \cdot \rvert$ denotes the coordinate-wise absolute value map.

\begin{thm}
\label{Thm::RojasRusek_Adiscriminant}
 \cite[Theorem 3.8.]{RojasRusek_Adiscriminant}
Let $(\cA,\eps)$ be a full-dimensional signed support with Gale dual matrix $B$ and let $c,c' \in \R^{\cA}_\eps$. If $B^\top\Log\lvert c \rvert $ and $B^\top\Log\lvert c' \rvert $ are in the same connected component of 
\[ \R^{k} \setminus \left( \bigcup_{F \subseteq \Conv( \cA )  \text{ a face}}  B^\top\Log\lvert \tilde{\nabla}_{\cA_F,\eps_F} \rvert \right), \]
then the zero sets $Z(f_c)$ and $Z(f_{c'})$ are  ambiently isotopic in $\R^{n}$.
\end{thm}

\begin{remark}
    \label{Remark:InnerChambers}
    We call a connected component $C$ of $\R^k \setminus \Gamma_\eps(A,B)$ a \emph{signed reduced outer chamber} if there exists $h \in \mathbb{R}^{\mathcal{A}}$ that induces a regular triangulation of $\mathcal{A}$, such that for all sufficiently large $s \in \mathbb{R}$ we have $B^\top(s h)\in C$. In other words, signed reduced outer chambers are precisely the connected components of $\R^k \setminus \Gamma_\eps(A,B)$ where the isotopy type of the corresponding hypersurfaces can be described using Viro's patchworking  (cf. Theorem \ref{Thm:Viro}).

    Otherwise, we call a connected component of  $\R^k \setminus \Gamma_\eps(A,B)$ a \emph{signed reduced inner chamber}. Note that signed reduced inner chambers may not be bounded.
\end{remark}
The signed reduced $A$-discriminant admits a parametrization as well.

\begin{prop}
    Let $(\cA,\eps)$ be a full-dimensional signed support with Gale dual matrix $B$. The image of the map \begin{align}
\label{Eq_SignedRedParam}
 \xi_{B,\eps}\colon \mathcal{C}_{B,\eps} \to \R^k, \quad \lambda \mapsto B^\top \Log \lvert B\lambda \rvert
 \end{align}
 is the signed reduced $A$-discriminant $\Gamma_\eps(A,B)$.
\end{prop}

\begin{proof}Let $\pr_k \colon \mathcal{C}_{B,\eps} \times \R^n \to \mathcal{C}_{B,\eps}$ be the natural coordinate projection and $\psi$ be the signed Horn-Kapranov Uniformization \eqref{Eq::HornKapranov}. Furthermore, denote $\tilde{A}$ the matrix obtained from $\hA$ in \eqref{Eq::AMatrix} by removing its top row, that is, the columns of $\tilde{A}$ are given by the vectors in $\cA$. For every $(\lambda,x) \in \mathcal{C}_{B,\eps} \times \R^n$, we have
\[\psi(\lambda,x) = B^\top \Log \lvert B\lambda \ast (e^{\alpha_i \cdot (-x)})_{i=1,\dots,n+k+1}\rvert = B^\top \Log \lvert B\lambda \rvert - B^\top \tilde{A}^\top x =  \xi_{B,\eps}(\lambda),\]
where the last equality holds, since $\tilde{A} B = 0$. Thus, the following diagram commutes
\begin{equation}
\label{Eq_ParamDiagram}
\begin{tikzcd}
\mathcal{C}_{B,\eps} \times \R^n  \arrow[d,"\pr_k"]  \arrow[r,"\psi"] & \R^{n+k+1}_\eps \arrow[r,"\Log \lvert \cdot \rvert"]  & \R^{n+k+1} \arrow[d,"B^\top"]  \\
\mathcal{C}_{B,\eps}   \arrow[rr,"\xi_{B,\eps}"] & & \R^k
\end{tikzcd}
\end{equation}
which gives that
\[ \im(\xi_{B,\eps}) = \im(\xi_{B,\eps} \circ \pr_k) = \im(B^\top\circ \Log\lvert \cdot \rvert \circ \psi) = B^\top\Log\lvert \nabla_{\cA,\eps} \rvert = \Gamma_\eps(A,B), \]
where the first equality holds since $\pr_k$ is surjective, and the second-to-last equality follows from Proposition \ref{Prop_HornKapranov}.
\end{proof}

Since the first row of the matrix $\hA$ is given by the all one vector $\mathds{1} \in \R^{n+k+1}$, we have $B^\top \mathds{1} = 0$, which implies that the map $\xi_{B,\eps}$ is homogeneous, i.e. for all $a \in \R$:
\[\xi_{B,\eps}(a\lambda) = B^\top \Log \lvert B ( a\lambda)  \rvert=\log(\lvert a \rvert) B^\top \mathds{1}+ B^\top \Log \lvert B\lambda \rvert = B^\top \Log \lvert B\lambda \rvert = \xi_{B,\eps}(\lambda) .\]
Thus, one could projectivize the domain $\mathcal{C}_{B,\eps} \subseteq \R^{k}$ of $\xi_{B,\eps}$. 

\begin{remark}
\label{Remark:FixB}
If the last row of the Gale dual matrix $B$ is not zero, then using elementary column operations one can assume without any restriction that the last row of $B$ has the form $B_{n+k+1} = (0,\dots,0,-1)$. Such a choice of $B$ fixes the sign $\eps_{n+k+1} = -1$. Since $Z(f_c) = Z(f_{-c})$, we can always fix one of the signs of the coefficients.

Since $\xi_{B,\eps}$ is homogeneous, one can replace $\R^k$   (assuming  $B_{n+k+1} = (0,\dots,0,-1)$) by the upper half of the $(k-1)-$sphere
\[C_{k-1} = \{ \lambda \in \R^k \mid \lVert \lambda \rVert = 1,\, \lambda_k > 0 \},\]
or by the $(k-1)-$dimensional affine subspace
\[  \{ \lambda \in \R^k \mid \lambda_k = 1 \}. \]
In Section \ref{Section_NoInnerChamber}, we will prefer this latter choice and work with the map
\begin{align}
\label{Eq_AffineSignedRedParam}
 \bar{\xi}_{B,\eps}\colon  \left\{ \mu \in \R^{k-1} \mid  \sign(B \begin{bmatrix}\mu \\  1 \end{bmatrix}) = \eps \right\} \to \R^k, \quad \mu \mapsto B^\top \Log \lvert B \begin{bmatrix}\mu \\  1 \end{bmatrix} \rvert.
 \end{align}
 The case when the last row of $B$ is zero will be discussed in more detail in  Assumption~\ref{Assumption}.
\end{remark}

\begin{ex}
\label{Ex:GameChanger}
    To give an illustration of the signed reduced $A$-discriminant, we recall \cite[Example 2.9]{forsgard2017new}. Consider the same set of exponent vectors as in Example \ref{Example_Viro}
    \begin{align}
\label{Eq:GameChangerSupport} \cA = \{ \alpha_1, \, \alpha_2, \, \alpha_3, \, \alpha_4,\, \alpha_5\} = \left\{
  \begin{bmatrix} 0\\ 0 \end{bmatrix},  \, \begin{bmatrix} 1 \\ 0 \end{bmatrix},   \, \begin{bmatrix} 0 \\ 1 \end{bmatrix} ,   \, \begin{bmatrix} 4 \\ 1 \end{bmatrix} ,   \,  \begin{bmatrix} 1 \\ 4 \end{bmatrix} \right\}.
  \end{align}
    Unlike in Example \ref{Example_Viro}, here we consider the sign distribution $\eps = (-1,1,1,-1,-1)$. We depicted the signed support $(\cA,\eps)$ in Figure~\ref{FIG_Example_2_1}(a). Since $\cA$ has codimension $2$, the signed reduced $A$-discriminant $\Gamma_\eps(A,B)$ is in the plane $\R^2$. For an illustration we refer to Figure~\ref{FIG_Example_2_1}(b). The complement of $\Gamma_\eps(A,B)$  has $3$ connected components. For the coefficient vectors $c = (-1,6,3,-1,-1)$,$(-1,1,1,-1,-1)$,$(-1,0.5,1,-1,-1)$, their projection $B^\top\Log\lvert c \rvert$ lies in different connected components of $\R^2 \setminus \Gamma_\eps(A,B)$. The corresponding hypersurfaces $Z(f_c)$ are shown in Figure~\ref{FIG_Example_2_1}(c),(d),(e) respectively. 

    One remarkable property of this particular signed support is that the isotopy type of the hypersurface in Figure~\ref{FIG_Example_2_1}(d), corresponding to the coefficient $c=(-1,1,1-1,-1)$, cannot be obtained by Viro's patchworking (cf. Theorem \ref{Thm:Viro}). All possible signed tropical hypersurfaces $\Trop_\eps(\cA,h)$, with $h\in \R^{\cA}$ generic, consist of  $2$ unbounded connected components by \cite[Theorem 4.5]{BihanBound}, but the hypersurface $Z(f_c), \,  c=(-1,1,1-1,-1)$ has $3$ connected components, one bounded and two unbounded. 
\end{ex}

\begin{figure}[t]
\centering
\begin{minipage}[h]{0.4\textwidth}
\centering
\includegraphics[scale=0.42]{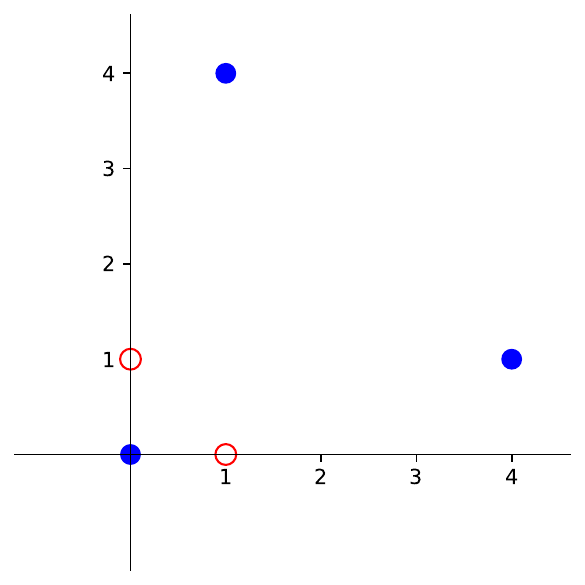}

{\small (a)}
\end{minipage}
\begin{minipage}[h]{0.4\textwidth}
\centering
\includegraphics[scale=0.3]{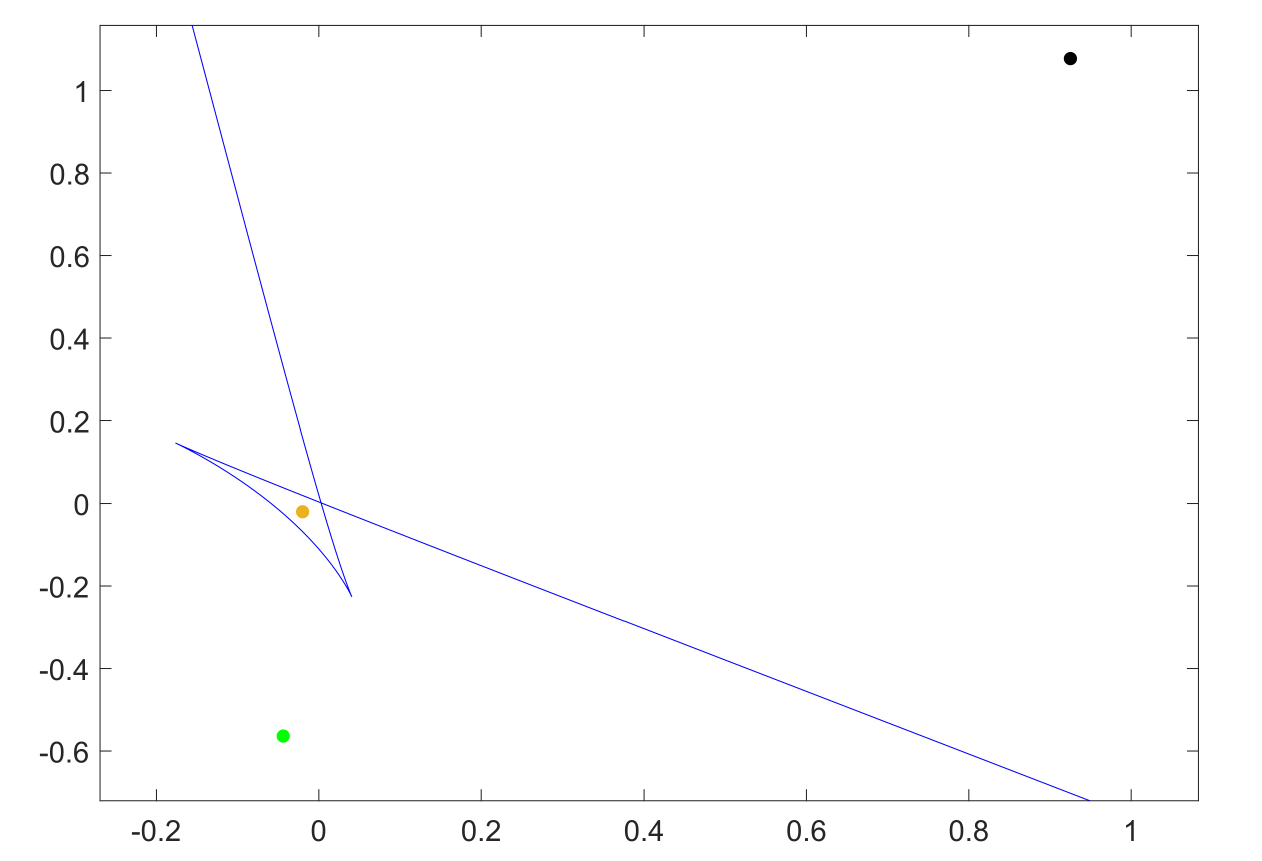}

{\small (b)}
\end{minipage}

\begin{minipage}[h]{0.3\textwidth}
\centering
\includegraphics[scale=0.35]{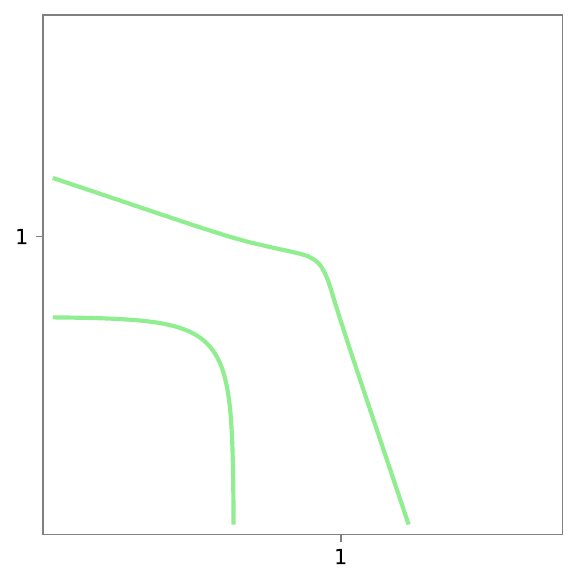}

{\small (c)}
\end{minipage}
\begin{minipage}[h]{0.3\textwidth}
\centering
\includegraphics[scale=0.35]{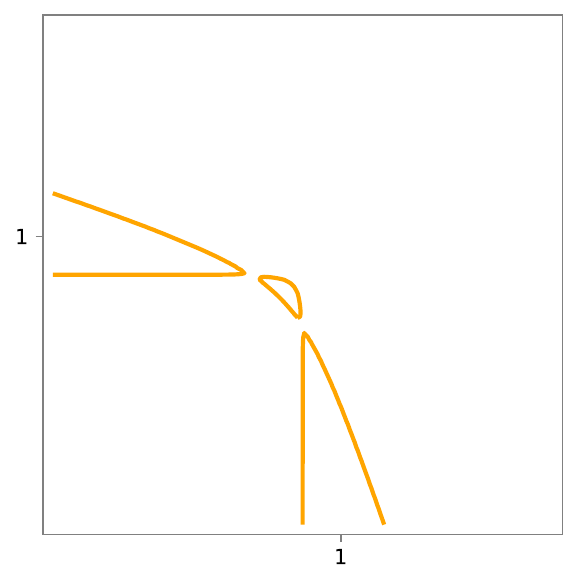}

{\small (d)}
\end{minipage}
\begin{minipage}[h]{0.3\textwidth}
\centering
\includegraphics[scale=0.35]{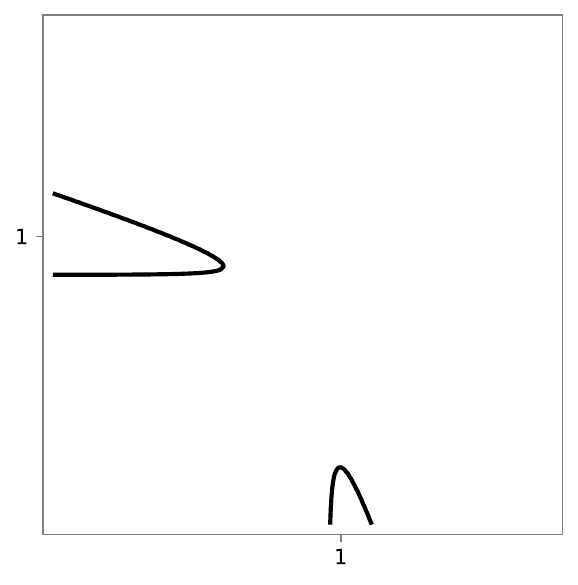}

{\small (e)}
\end{minipage}
\caption{{\small (a) Signed support from Example \ref{Ex:GameChanger}. The positive and negative exponent vectors are depicted by red circles and blue dots respectively. (b) Signed reduced $A$-discriminant of the signed support in (a). (c)(d)(e) Hypersurfaces $Z(f_c)$ corresponding to different connected components of the complement of the signed reduced $A$-discriminant. }}\label{FIG_Example_2_1}
\end{figure}

\subsection{Some useful results from topology}
\label{Sec:ResFromTop}
In the proof of Proposition \ref{Prop_NoCrit_NoInnerChamber}, we need some classical results from topology. Let us also introduce these briefly here (see, for example, Chapter \uppercase\expandafter{\romannumeral1}.11 \& Chapter \uppercase\expandafter{\romannumeral4}.19 in \cite{Bredon_book}).

\begin{lemma}
\label{lem:proper}
  Suppose that $X$ and $Y$ are locally compact, Hausdorff spaces and that $f\colon X\to Y$ is continuous. Let $X^+:=X\sqcup \{\infty_X\}$ be the one-point compactification space of $X$. Then $f$ is proper (i.e., the preimage of any compact subset is compact) $\Longleftrightarrow$ $f$ extends to a continuous map $f^+\colon X^+\to Y^+$ by setting $f^+(\infty_X)=\infty_Y$.
\end{lemma}

\begin{lemma}(Jordan-Brouwer Separation Theorem)\label{lem:Sep}
  If $f\colon \mathbf{S}^{n-1}\to \mathbf{S}^{n}$ (where $\mathbf{S}^{n}$ denotes $n$-sphere) is an injective continuous map, then $\mathbf{S}^{n} \setminus f(\mathbf{S}^{n-1})$ consists of exactly two connected components. Moreover, $f(\mathbf{S}^{n-1})$ is the topological boundary of each of these components.
\end{lemma}

\begin{cor}
\label{cor:unbounded components}
  If $f\colon \R^n\to \R^{n+1}$ is injective, continuous and proper, then $\R^{n+1} \setminus f(\R^n)$ consists of exactly two unbounded connected components.
\end{cor}
\begin{proof}
  Note that the one point compactification of $\R^n$ is $\mathbf{S}^{n}$. By Lemma \ref{lem:proper}, $f$ can be extended to $f^+\colon \mathbf{S}^{n}\to \mathbf{S}^{n+1}$ with $f^+(\infty)=\infty$. Then $f^+$ is also injective. By Lemma \ref{lem:Sep}, $\mathbf{S}^{n+1} \setminus f^+(\mathbf{S}^{n})$ consists of two connected components and the point $\infty$ is in the boundary. Since the point $\infty$ is in the boundary of each component, we can always find a sequence $\{x_l\}\subseteq \R^{n+1}=\mathbf{S}^{n+1}\setminus \{\infty\}$ in each of the components such that $x_l\to \infty$. Therefore, these two components are unbounded in $\R^{n+1}$.
\end{proof}

Finally, we need the following version of the mean value theorem:

\begin{lemma}(Cauchy's Mean Value Theorem)
\label{lem:CMVT}
If the functions $f, g \colon [a,b] \to \R$ are both continuous and differentiable on the open interval $(a,b)$, then there exists some $c\in (a,b)$, such that
\[
(f(b)-f(a))g'(c)=(g(b)-g(a))f'(c).
\]
\end{lemma}

\section{Signed supports without singular zeros}

In this section, we give a necessary and sufficient condition on the signed support $(\cA,\eps)$ such that the signed $A$-discriminant $\nabla_{\cA,\eps}$ is empty.  Building on this result and  Theorem \ref{Thm::RojasRusek_Adiscriminant}, we give conditions on $(\cA,\eps)$ such that for all $c \in \R_\eps^\cA$ the hypersurfaces $Z(f_c)$ have the same isotopy type (Theorem \ref{Thm::SepHyperplane}). First, we start with a simple observation.

\begin{prop}
\label{Prop_SolAndKernel}
Let $\cA = \{ \alpha_1, \dots, \alpha_{n+k+1} \} \subseteq \R^n$ be a set of exponent vectors and $\eps \in \{ \pm 1 \}^{n+k+1}$ be a fixed sign distribution. Let $\hA$ the matrix defined in \eqref{Eq::AMatrix}.
\begin{itemize}
\item[(a)] If $\ker( \hA) \cap \R_{\eps}^{n+k+1} = \ker( \hA_\eps) \cap \R_{>0}^{n+k+1} = \emptyset$, then for all $c \in \R_{\eps}^{n+k+1}$ the critical system \eqref{Eq_FullCriticalSystem} does not have any solution $x \in \R^{n}$.
\item[(b)] If $\ker( \hA_\eps) \cap \R_{>0}^{n+k+1} \neq \emptyset$, then there exists $c \in \R_{\eps}^{n+k+1}$ such that the critical system has a solution $x \in \R^{n}$.
\end{itemize}
\end{prop}

\begin{proof} Part (a) follows directly, since for any $c \in \R^{n+k+1}_{\eps}$ and any solution $x\in \R^n$ of  \eqref{Eq_FullCriticalSystem}, we have $\diag(c)\big(e^{\alpha_i \cdot x}\big)_{i=1,\dots,n+k+1} \in \ker(\hA) \cap \R_{\eps}^{n+k+1}$.

If $v \in \ker( \hA_\eps) \cap \R_{>0}^{n+k+1}$, then for $c=\diag(\eps)  v$, the point $x = (0,\dots,0)$ is a solution of \eqref{Eq_FullCriticalSystem}.
\end{proof}

In the following, we interpret the conditions in Proposition \ref{Prop_SolAndKernel} in terms of the geometry of the support $\cA$ and the sign distribution $\eps$. An \emph{affine hyperplane} is a set of the form
\[ \mathcal{H}_{v,a} := \{ \mu \in \R^n \mid v \cdot \mu = a \},\]
for some $v \in \R^n \setminus \{ 0 \}$ and $a \in \R$. Each affine hyperplane defines two half-spaces
\[ \mathcal{H}_{v,a}^+ := \{ \mu \in \R^n \mid v \cdot \mu \geq a \}, \quad \mathcal{H}_{v,a}^- := \{ \mu \in \R^n \mid v \cdot \mu \leq a \}.\]
Following \cite{DescartesHypPlane}, we call $ \mathcal{H}_{v,a}$ a \emph{separating hyperplane} of $(\cA,\eps)$ if
\begin{align}
\label{Eq_SepVectorIn}
\cA_+ \subseteq  \mathcal{H}_{v,a}^+, \quad \text{and} \quad  \cA_- \subseteq  \mathcal{H}_{v,a}^-.
\end{align}
A separating hyperplane $ \mathcal{H}_{v,a}$ is called \emph{non-trivial}, if at least one of the open half-spaces $\inte(\mathcal{H}_{v,a}^{+}), \, \inte(\mathcal{H}_{v,a}^{-})$ contains a point of $\cA$. A non-trivial separating hyperplane is called  \emph{very strict} if $\mathcal{H}_{v,a}$ does not contain any point in $\cA$. For an illustration of separating hyperplanes, we refer to Figure \ref{SepHyperplane}.

\begin{figure}[t]
\centering
\includegraphics[scale=0.5]{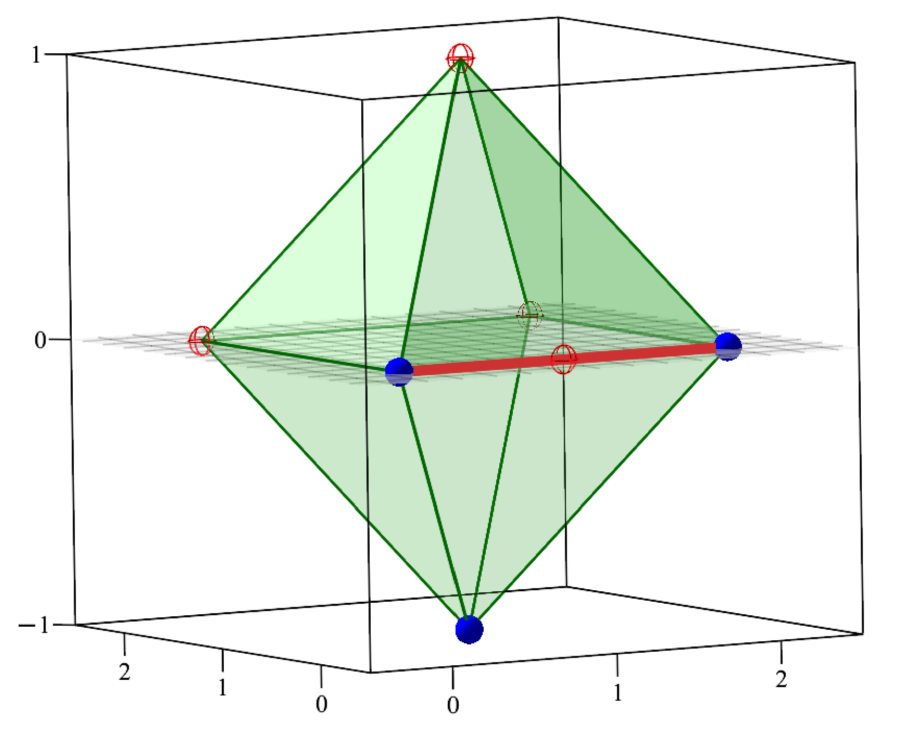}
\caption{{\small The hyperplane $\mathcal{H}_{v,0}$ with $v = (0,0,1)$ is a non-trivial separating hyperplane of $\cA_+ = \big\{ (1,0,0)^\top,(2,2,0)^\top,(0,2,0)^\top,(1,1,1)^\top\big\}$ (depicted as red circles) and $\cA_- = \big\{ (0,0,0)^\top,(2,0,0)^\top,(1,1,-1)^\top\big\}$ (blue dots). For the face $F = \Conv\big( (0,0,0)^\top,(0,0,2)^\top \big)$ (marked by thick line segment), the restricted signed support $\cA_{F,+} = \big\{ (1,0,0)^\top\big\},\, \cA_{F,-} = \big\{(0,0,0)^\top,(2,0,0)^\top \big\}$ does not have any non-trivial separating hyperplane.
}}\label{SepHyperplane}
\end{figure}

\begin{prop}
\label{Prop_SepHypAndKernel} A signed support $(\cA, \eps)$ has a non-trivial separating hyperplane if and only if $\ker( \hA_\eps) \cap \R_{>0}^{n+k+1} = \emptyset$, where $\hA \in \R^{(n+1) \times (n+k+1)}$ denotes the matrix from \eqref{Eq::AMatrix}.
\end{prop}

\begin{proof}
By Stiemke's Theorem \cite{Stiemke1915berPL} (which is a refinement of Farkas' Lemma, see also \cite[Section 6.2]{Ziegler_book}), exactly one of the following holds. Either there exists $w \in \R^{n+1}$ such that
\begin{align}
\label{Eq_SepHypProof1}
\hA_\eps^\top w \geq 0,
\end{align}
and at least one of the inequalities is strict, or there exists $u \in \R_{>0}^{n+k+1}$ such that
\begin{align}
\label{Eq_SepHypProof2}
\hA_\eps u = 0.
\end{align}
(The conditions \eqref{Eq_SepHypProof1} and \eqref{Eq_SepHypProof2} can not both hold.) Condition \eqref{Eq_SepHypProof2} is equivalent to $\ker(\hA_\eps) \cap  \R_{>0}^{n+k+1} \neq \emptyset$.  Note that one can rewrite  \eqref{Eq_SepHypProof1} as
\[ \eps_i ((w_2, \dots, w_{n+1}) \cdot \alpha_i) \geq \eps_i(-w_{1}),\]
for all $\alpha_i \in \cA$. Thus, if such a $w$ exists, then $\mathcal{H}_{v,a}$ with $v=(w_2, \dots ,w_{n+1}), \, a = -w_{1}$ is a non-trivial separating hyperplane of $(\cA,\eps)$. On the other hand, if $\mathcal{H}_{v,a}$ is a non-trivial hyperplane of  $(\cA,\eps)$, then $w = (-a,v)$ satisfies \eqref{Eq_SepHypProof1}.
\end{proof}

\begin{thm}
\label{Thm::SepHyperplaneAdisc}
Let $(\cA,\eps)$ be a signed support with Gale dual matrix $B$. Then the following are equivalent:
\begin{itemize}
    \item[(i)] $\nabla_{\cA,\eps} = \emptyset$
    \item[(ii)]  $\Gamma_\eps(A,B) = \emptyset$
    \item[(iii)] $(\cA,\eps)$ has a non-trivial separating hyperplane.
\end{itemize}
\end{thm}

\begin{proof}
  The equivalence between (i) and (ii) follows directly from the definition of the signed reduced $A$-discriminant, since $\Gamma_\eps(A,B) = B^\top\Log\lvert \nabla_{\cA,\eps}\rvert$. From Proposition \ref{Prop_SolAndKernel} it follows that $\nabla_{\cA,\eps} = \emptyset$ if and only if $\ker( \hA_\eps) \cap \R_{>0}^{n+k+1} = \emptyset$, which is equivalent to the existence of a non-trivial separating hyperplane of $(\cA, \eps)$ by Propositon \ref{Prop_SepHypAndKernel}. This shows that (i) and (ii) are equivalent.
\end{proof}

For fixed set of exponent vectors $\cA \subseteq \R^n$, using the correspondence between hyperplane arrangements and zonotopes, one derives a bound on the number of sign distributions for which $(\cA, \eps)$ does not have a non-trivial separating hyperplane.

\begin{prop}
\label{Prop_NumOfNonTrivSepHyp}
Let $\cA = \{ \alpha_1, \dots , \alpha_{n+k+1} \}  \subseteq \R^n$ be a finite set such that $\dim \Conv(\cA) = n$. The number of sign distributions $\eps \in \{ \pm 1 \}^{n+k+1}$ for which $(\cA,\eps)$ does not have a non-trivial separating hyperplane is bounded above by:
\[ 2 \sum_{i=0}^{k-1} \binom{n+k}{i}.\]
\end{prop}

\begin{proof}
Let $\hA$ be the matrix defined in \eqref{Eq::AMatrix}. By Proposition \ref{Prop_SepHypAndKernel}, the signed support $(\cA, \eps)$ does not have a non-trivial separating hyperplane if and only if $\ker( \hA) \cap \R_{\eps}^{n+k+1} \neq \emptyset$. So all we have to do is to count how many orthants $\R_{\eps}^{n+k+1}$ $\ker( \hA)$ intersects.

The assumption $\dim \Conv(\cA) = n$ implies that $\dim \ker(\hA) = k$. Let $B \in \R^{(n+k+1) \times k}$ be Gale dual to $\hA$ and denote by $B_1, \dots , B_{n+k+1}$  the rows of $B$.  By \cite[Lemma 0.16]{Fukuda2004LectureNO} (see also  \cite[Corollary 7.17]{Ziegler_book}), the orthants  $\R_{\eps}^{n+k+1}$ that $\im(B) = \ker( \hA)$ intersects, correspond one-to-one to the vertices of the zonotope
\[ [-B_1,B_1] + \dots + [-B_{n+k+1},B_{n+k+1}]  \subseteq \R^{k}. \]
By \cite[Table 2.1]{FukudaNotes}  (see also \cite{zaslavsky1975facing}) such a zonotope can have at most
\[ 2 \sum_{i=0}^{k-1} \binom{n+k}{i}\]
many vertices.
\end{proof}

We finish the section by providing a condition that ensures all hypersurfaces corresponding to a given signed support have the same isotopy type.

\begin{thm}
\label{Thm::SepHyperplane}
Let $(\cA,\eps)$ be a signed support. If for all faces $F \subseteq \Conv(\cA)$ the signed support $(\cA_F,\eps_F)$ has a non-trivial separating hyperplane, then for all $c \in \R^\cA_\eps$ the hypersurfaces $Z(f_c)$ have the same isotopy type.
\end{thm}

\begin{proof}
 From Theorem \ref{Thm::SepHyperplaneAdisc}  follows that the signed reduced $A$-discriminants associated to the faces  $F \subseteq \Conv(\cA)$ are empty. Thus, all the hypersurfaces $Z(f_c), \, c \in \R_\eps^\cA$ have the same isotopy type by Theorem  \ref{Thm::RojasRusek_Adiscriminant}.
\end{proof}

\begin{cor}
\label{Cor:SepHyp}
If a signed support $(\cA,\eps)$ has a very strict separating hyperplane, then the hypersurfaces $Z(f_c)$ have the same isotopy type for all $c \in \R^\cA_\eps$.
\end{cor}

\begin{proof}
If $\mathcal{H}_{v,a}$ is a very strict separating hyperplane of $(\cA,\eps)$, then it is also a very strict separating hyperplane of $(\cA_F,\eps_F)$ for all faces $F \subseteq \Conv(\cA)$. Now, the statement follows from Theorem \ref{Thm::SepHyperplane}.
\end{proof}

\begin{ex}
    The signed support $(\cA,\eps)$ from Example \ref{Example_Viro} has a very strict separating hyperplane. Thus, by Corollary \ref{Cor:SepHyp}, all hypersurfaces $Z(f_c), \, c \in \R^\cA_\eps$ have the same isotopy type. From Theorem \ref{Thm:Viro} follows that this isotopy type agrees with the isotopy type of the signed tropical hypersurface $\Trop_\eps(\cA,h)$ for every generic $h \in \R^\cA$. We refer to Figure \ref{FIG_Example_Viro}(b) for an illustration of $\Trop_\eps(\cA,h)$ with $h=(5,4,4,5,5)$.
\end{ex}

\begin{remark}
If a signed support $(\cA,\eps)$ has a non-trivial separating hyperplane, it might happen that for one of the faces $F \subseteq \Conv(\cA)$ the restricted signed support $(\cA_F,\eps_F)$ does not have a non-trivial separating hyperplane. For such an example, we revisit the signed support from Figure \ref{SepHyperplane}. The face $F = \Conv( (0,0,0)^\top,(2,0,0)^\top)$, contains two negative exponent vectors $\alpha_1 = (0,0,0)^\top, \alpha_2 = (2,0,0)^\top$ and one positive exponent vector $\alpha_3 = (1,0,0)^\top$. Since $\alpha_3$ lies in the relative interior of $\Conv(\alpha_1,\alpha_2)$, it follows that $\{\alpha_1,\alpha_2\}$ and $\{\alpha_3\}$ cannot be separated by an affine hyperplane.
\end{remark}

\section{A-discriminants with two signed reduced outer chambers}
\label{Section_NoInnerChamber}

The goal of this section is to describe conditions on the signed support $(\cA,\eps)$ that ensure $\R^k \setminus \Gamma_\eps(A,B)$ has at most two connected components, and each is unbounded. In Section~\ref{Section::Cusps}, we focus on the case where $\cA$ has exactly $n+3$ exponent vectors. Under this assumption, the signed reduced $A$-discriminant $\Gamma_\varepsilon(A,B)$ is a curve in the plane, and we apply classical results (cf. Section~\ref{Sec:ResFromTop}) to analyze its topology. We show that the complement of $\Gamma_\eps(A,B)$ has at most two chambers if the parametrization map $\overline{\xi}_{B,\eps}$ has at most one critical point (Proposition \ref{Prop_NoCrit_NoInnerChamber}). It is known that $\overline{\xi}_{B,\eps}$ can have at most $n$ critical points \cite{Rusek::Thesis}, however there did not exist any known example in the literature where this bound is attained. In Example~\ref{Eq::4cusps}, we describe a family of signed supports such that $\overline{\xi}_{B,\eps}$  has $n$ critical points for every $n \in \mathbb{N}$.

In Section \ref{Section::DegenSingularPoints}, we investigate the relation between critical points of $\overline{\xi}_{B,\eps}$ and degenerate singular points of  $Z(f_c)$ , and show that if $Z(f_c)$ has no degenerate singular point for all $c \in \R^\cA_\eps$ and the codimension of $\cA$ is $2$, then $\R^k \setminus \Gamma_\eps(A,B)$ has at most two connected components  (Theorem \ref{Thm_NoInnerChamber}). In Section \ref{Section::NoDegenerateSingPoints}, we give several conditions on the geometry of the signed support $(\cA,\eps)$ precluding the existence of degenerate singular points in $Z(f_c), \, c \in \R^\cA_\eps$.

\begin{assumption}
\label{Assumption}
During the whole section, we assume that we can choose the Gale dual matrix $B \in \R^{(n+k+1)\times k}$ such that its last row has the form $(0,\dots,0,-1)$. By Remark \ref{Remark:FixB}, the only case when we cannot make this assumption is when the last row of $B$ is zero. This is equivalent to saying that $\mathcal{A}$ is a pyramid with apex at $\alpha_{n+k+1}$ (cf. \cite[Section 2.2]{BihanBound}). In this case, the signed reduced $A$-discriminant is empty by Theorem~\ref{Thm::SepHyperplaneAdisc}. Therefore, we can assume, without loss of generality, that the last row of $B$ is  $(0,\dots,0,-1)$.
\end{assumption}

\subsection{Critical points of the signed reduced A-discriminant}
\label{Section::Cusps}

Let $\cA = \{ \alpha_1, \dots , \alpha_{n+k+1}\} \subseteq \R^n$ be a set of exponent vectors such that $\dim \Conv(\cA) = n$ and fix a sign distribution $\eps \in \{ \pm 1 \}^{n+k+1}$. Let $\hA \in \R^{(n+1) \times (n+k+1)}$ be as given in \eqref{Eq::AMatrix} and choose a Gale dual matrix $B \in \R^{(n+k+1)\times k}$ with rows $B_1, \dots, B_{n+k+1}$ such that its last row has the form $B_{n+k+1} = (0, \dots , 0 ,-1)$.

Let $\bar{\xi}_{B,\eps}$ be the parametrization map of $\Gamma_\eps(A,B)$ as defined in \eqref{Eq_AffineSignedRedParam}. Following \cite[Section 1.2]{arnold2012singularities}, we call a point $\mu \in \R^{k-1}$ a \emph{critical point} of $\bar{\xi}_{B,\eps}$ if $\bar{\xi}_{B,\eps}(\mu)$ is well-defined and the Jacobian matrix  $J_{\bar{\xi}_{B,\eps}}(\mu)$ does not have full rank, that is, it has rank strictly less than $k-1$.

\begin{lemma}
\label{Lemma_BJac}
    Let $(\cA,\eps)$ be a signed support with Gale dual matrix $B$, and let $\bar{\xi}_{B,\eps}$ be as defined in \eqref{Eq_AffineSignedRedParam}. Then for each $\mu \in \R^{k-1}$ where  $\bar{\xi}_{B,\eps}$ is defined, we have the following equality for the Jacobian matrix
\begin{align}
\label{Eq::JacobianAffineParam}
J_{\bar{\xi}_{B,\eps}}(\mu) = B^\top \diag\Big( \tfrac{1}{B \hat{\mu}} \Big)\tilde{B}  ,
\end{align}
where $\tilde{B}$ denotes the matrix obtained from  $B$ by deleting its last column and $\hat{\mu}= \begin{bmatrix}\mu \\  1 \end{bmatrix}$.
\end{lemma}

\begin{proof}
    By definition (cf. \eqref{Eq_AffineSignedRedParam}), the $j-$th coordinate of $\bar{\xi}_{B,\eps}$ is given by
\[ \big( \bar{\xi}_{B,\eps}(\mu) \big)_j  = \sum_{i =1}^{n+k+1} \log \lvert B_i \cdot \hat{\mu} \rvert B_{i,j},\]
where $\hat{\mu} = \begin{bmatrix}\mu \\  1 \end{bmatrix}$. Thus, the partial derivatives of $\bar{\xi}_{B,\eps}$ have the form
\begin{align}
\label{Eq:LemmaBJac}
     \frac{\partial\big( \bar{\xi}_{B,\eps}(\mu) \big)_j  }{\partial \mu_{\ell}} = \sum_{i=1}^{n+k+1} \frac{B_{i,j}  B_{i,\ell}}{B_i \cdot \hat{\mu}}
\end{align}
for all $j = 1, \dots , k$,  and $\ell = 1, \dots , k-1$. Comparing \eqref{Eq:LemmaBJac} with the entries of the right-hand side of \eqref{Eq::JacobianAffineParam} the result follows.
\end{proof}

Using Lemma \ref{Lemma_BJac}, an easy computation shows that 
\begin{align}
\label{Eq::KapranovNormal}
 \hat{\mu}^\top J_{\bar{\xi}_{B,\eps}}(\mu) = \Big( B \hat{\mu} \Big)^\top \diag\Big( \tfrac{1}{B \hat{\mu}} \Big) \tilde{B} = \mathds{1}^\top \tilde{B} = 0. 
 \end{align}
Therefore, if $\bar{\xi}_{B,\eps}$ is differentiable at $\mu$, then a normal vector at $\bar{\xi}_{B,\eps}(\mu)$ is given by $\hat{\mu}$. This statement was proven by Kapranov \cite[Theorem 2.1]{Kapranov1991ACO}.

In the remaining of the section, we focus on the case $k=2$. Under this assumption, we have 
\[ J_{\bar{\xi}_{B,\eps}}(\mu) = \begin{bmatrix} b_1^\top\diag\Big( \tfrac{1}{B \hat{\mu}} \Big) b_1 \\ b_2^\top\diag\Big( \tfrac{1}{B \hat{\mu}} \Big) b_1 \end{bmatrix} ,\]
where $b_1, b_2$ denote the first and second column of the Gale dual matrix $B$. 

 \begin{lemma}
 \label{Lemma:NumCritPoints}
  Let $(\cA,\eps)$ be a signed support of codimension $2$ with Gale dual $B \in \R^{(n+3) \times 2}$, and let $\bar{\xi}_{B,\eps}$ be as defined in \eqref{Eq_AffineSignedRedParam}. For $\mu \in \R \setminus \{0\}$ the following are equivalent.
  \begin{itemize}
      \item[(i)] $\mu$ is a critical point of $\bar{\xi}_{B,\eps}$.
      \item[(ii)] $\sign( B\hat{\mu}) = \eps$ and $\mu$ is a zero of the univariate polynomial
\begin{align}
\label{Eq::CritPoly}
 q_B(\mu) := \Big(\prod_{i = 1}^{n+3}  (B\hat{\mu})_i\Big) b_1^\top  \diag\Big( \tfrac{1}{B \hat{\mu}} \Big) b_1.
 \end{align}
      \item[(iii)] $\sign( B\hat{\mu}) = \eps$ and $\mu$ is a zero of the univariate polynomial
\begin{align}
\label{Eq::CritPoly2}
 \tilde{q}_B(\mu) := \Big(\prod_{i = 1}^{n+3}  (B\hat{\mu})_i\Big) b_2^\top  \diag\Big( \tfrac{1}{B \hat{\mu}} \Big) b_1.
 \end{align}
  \end{itemize}
 Moreover, $q_B$ has degree at most $n$ and $\bar{\xi}_{B,\eps}$ has at most $n$ critical points.
 \end{lemma}

\begin{proof}
Note that $\bar{\xi}_{B,\eps}(\mu)$ is defined only if $\sign( B\hat{\mu}) = \eps$. Furthermore, the factor $\prod_{i = 1}^{n+3}  (B\hat{\mu})_i$ clears the denominator of $b_1^\top  \diag\Big( \tfrac{1}{B \hat{\mu}} \Big) b_1$ and of $b_2^\top  \diag\Big( \tfrac{1}{B \hat{\mu}} \Big) b_1$. Moreover, $\prod_{i = 1}^{n+3}  (B\hat{\mu})_i \neq 0$ if $\sign( B\hat{\mu}) = \eps$. 
    From \eqref{Eq::KapranovNormal} follows that
    \[ \mu b_1^\top  \diag\Big( \tfrac{1}{B \hat{\mu}} \Big) b_1 = - b_2^\top  \diag\Big( \tfrac{1}{B \hat{\mu}} \Big) b_1.\]
    Thus,  $J_{\bar{\xi}_{B,\eps}}(\mu) = 0$ if and only if $q_B(\mu) = 0$, which is also equivalent to $\tilde{q}_B(\mu) = 0$.
    
The polynomial $q_B$ has been studied previously in \cite[Theorem 3.4]{Rusek::Thesis}, where it has been shown that its degree is at most $n$.
\end{proof}

In \cite[Theorem 3.10]{Rusek::Thesis}, the author constructed several matrices $B \in \R^{(n+3) \times 2}$ such that $ q_B(\mu)$ has exactly $n$ real roots.  However, these roots correspond to different sign distributions $\eps$. To show that the bound on the critical points of $\bar{\xi}_{B,\eps}$ in Lemma \ref{Lemma:NumCritPoints} is attained, one needs to construct $B \in \R^{(n+3) \times 2}$ such that $q_{B}$ (or $\tilde{q}_B$) has $n$ real roots $\mu_1, \dots, \mu_n$ such that $\sign( B\hat{\mu}_1) = \dots = \sign( B\hat{\mu}_n) = \eps$ for some fixed $\eps \in \{ \pm 1 \}^\cA$. We provide such a construction in the following example.

\begin{ex}
\label{Eq::4cusps}
Let $n\in \mathbb{N}$ and let $\mu_1,\ldots,\mu_n$ be distinct positive numbers different from $1$. Consider the univariate polynomials $f(\mu) = (\mu-\mu_1)(\mu-\mu_2)\cdots (\mu-\mu_n)$, $g(\mu) = (\mu+\mu_1)(\mu+\mu_2)\cdots(\mu+\mu_n)(\mu+1) \in \R[\mu]$. Since $\deg(f) = n < \deg(g) = n+1$, the fraction $\tfrac{f(\mu)}{g(\mu)}$ admits a partial fraction decomposition:
\begin{align}
\label{Eq::FracDecomp}
\frac{(\mu-\mu_1)(\mu-\mu_2)\cdots (\mu-\mu_n)}{(\mu+\mu_1)(\mu+\mu_2)\cdots(\mu+\mu_n)(\mu+1)}=\frac{a_1}{\mu+\mu_1}+\frac{a_2}{\mu+\mu_2}+\cdots+\frac{a_{n}}{\mu+\mu_n}+\frac{a_{n+1}}{\mu+1},
\end{align}
where $a_1, \dots , a_{n+1} \in \R$. The $a_i$'s satisfy the following properties:
\begin{itemize}
\item[(1)] ${\displaystyle \frac{a_1}{\mu_1}+\frac{a_2}{\mu_2}+\frac{a_3}{\mu_3}+\cdots+\frac{a_{n}}{\mu_n}+a_{n+1}=(-1)^n}$,
\item[(2)] $a_1+a_2+\cdots+a_{n+1}=1$.
\item[(3)] $a_i \neq 0$, $i = 1, \dots, n+1$.
\end{itemize}
Property (1) follows by plugging in $\mu=0$, (2) follows by comparing the leading coefficients of numerators on both sides, (3) follows by comparing the degree of denominators on both sides.

We use the $a_i$'s to build the matrix:
\[
B=\begin{bmatrix}
    \frac{a_1}{\mu_1} & \frac{a_2}{\mu_2} & \frac{a_3}{\mu_3} & \cdots & \frac{a_{n}}{\mu_n} &a_{n+1} & (-1)^{n+1} & 0 \\
    a_1 & a_2 & a_3 & \cdots & a_n & a_{n+1} & 0 & -1
  \end{bmatrix}^\top.
\]
Properties (1) and (2) imply that  $\mathds{1}^\top B = 0$, thus it is possible to choose a matrix $\hA \in \R^{(n+1)\times (n+3)}$ as in \eqref{Eq::AMatrix} such that $B$ is its Gale dual. Denoting by $b_1,b_2$ the columns of $B$, we have:
\[
b_2^\top\diag\Big( \tfrac{1}{B \hat{\mu}} \Big) b_1  = \frac{\frac{a_1^2}{\mu_1}}{\frac{a_1}{\mu_1}\mu+a_1}+\frac{\frac{a_2^2}{\mu_2}}{\frac{a_2}{\mu_2}\mu+a_2}+\frac{\frac{a_3^2}{\mu_3}}{\frac{a_3}{\mu_3}\mu+a_3}+\cdots+\frac{\frac{a_{n}^2}{\mu_n}}{\frac{a_{n}}{\mu_n}\mu+a_{n}}+\frac{a_{n+1}^2}{a_{n+1}\mu+a_{n+1}}.
\]
The right-hand side of this equality agrees with the right-hand side of \eqref{Eq::FracDecomp}, as the $a_i$'s are nonzero. Therefore, the zeros of $\tilde{q}_B(\mu) = b_2^\top\diag\Big( \tfrac{1}{B \hat{\mu}} \Big) b_1 $ are $\mu_1,\dots,\mu_n$. Since $\mu_1, \dots ,\mu_n$ are positive, it follows for all $i = 1, \dots , n$ that
\[ \sign (B \hat{\mu}_i ) = (\operatorname{sign}(a_1),\operatorname{sign}(a_2),\ldots,\operatorname{sign}(a_{n+1}),(-1)^{n+1},-1) =: \eps.\]
We conclude that $\bar{\xi}_{B,\eps}$ has $n$ critical points using Lemma \ref{Lemma:NumCritPoints}.
\medskip

\noindent
\begin{minipage}{0.6\textwidth}
Let $n=4$ and pick $\mu_1=5,\mu_2=6,\mu_3=7,\mu_4=8$. We have
\begin{align*}
   &\frac{(\mu-5)(\mu-6)(\mu-7)(\mu-8)}{(\mu+5)(\mu+6)(\mu+7)(\mu+8)(\mu+1)} \\
   =& \frac{-715}{\mu+5}+\frac{\frac{12012}{5}}{\mu+6}+\frac{-2730}{\mu+7}+\frac{1040}{\mu+8}+\frac{\frac{18}{5}}{\mu+1}
\end{align*}
and
\[
B=\begin{bmatrix}
    -143& \frac{2002}{5} & -390 & 130 & \frac{18}{5} & -1 & 0 \\
   -715 & \frac{12012}{5} & -2730 & 1040 & \frac{18}{5} & 0 & -1
  \end{bmatrix}^\top.
\]

By the above, the map $\bar{\xi}_{B,\eps}$ has $4$ critical points for $\eps =  (-1, 1, -1,1 ,1 ,-1 ,-1)$.  The signed reduced $A$-discriminant is drawn to the right and its critical points are highlighted by red circles.
\end{minipage}
\begin{minipage}{0.4\textwidth}
  \includegraphics[width=0.95\textwidth]{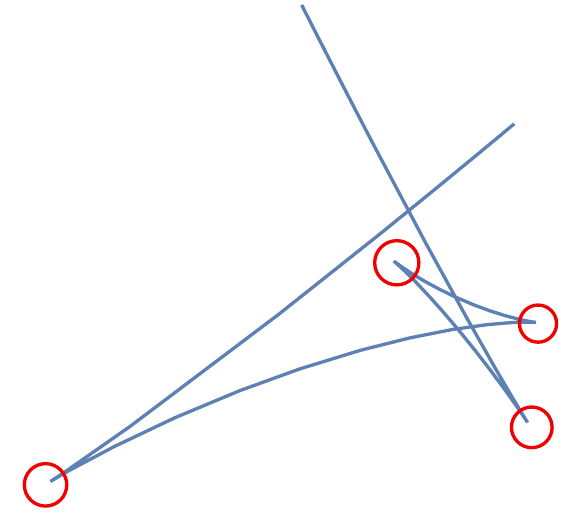}
\end{minipage}
\end{ex}

In Example \ref{Ex:GameChanger} (cf. Figure \ref{FIG_Example_2_1}(b)), we saw that the complement of $\Gamma_\eps(A,B)$ has three connected components if $\bar{\xi}_{B,\eps}$ has $2$ critical points. In the following, we show that if  $\bar{\xi}_{B,\eps}$ has at most one critical point, then $\R^k \setminus \Gamma_\eps(A,B)$ cannot have more than two connected components.

This is one of the few instances where working with the reduced version of the signed $A$-discriminant is crucial. If $\mathcal{A}$ has $n+3$ points, then the signed reduced $A$-discriminant is a curve, allowing us to apply the following result to study its self-intersections.

\begin{lemma}
\label{lem: self intersection}
  Let $\varphi\colon \R \to \R^{2}$ be a smooth map such that the Jacobian matrix $J_{\varphi}(\mu)$ has full rank for all $\mu \in \R$ except for at most one point. Let $S\subseteq \R^{2}$ be the curve parametrized by $\varphi$. If there exist two distinct points $a,b\in \R$ such that $\varphi(a)=\varphi(b)$, then there exist two distinct points $\mu_1,\mu_2\in \R$ such that the tangent lines of $S$ at $\varphi(\mu_1)$ and at $\varphi(\mu_2)$ are parallel.
\end{lemma}
\begin{proof}
Denote $\varphi_1,\varphi_2$ the first and the second coordinate of $\varphi$. Suppose $a<b$, and assume there exists $t \in \R$ such that $\varphi'_1(t)=\varphi'_2(t)=0$, that is, $J_\varphi(t)$ does not have full rank. We start by choosing $c \in \R$ such that $a<c<b$, $\varphi(c)\neq \varphi(a)$ and $\varphi$ is smooth on both intervals $(a,c)$ and $(c,b)$. If $t\leq a$ or $t\geq b$, such $c$ exists since the Jacobian matrix $J_{\varphi}$ has full rank on $(a,b)$. If $a<t<b$ and $\varphi(t)=\varphi(a)$, then the curve is smooth on the interval $(a,t)$, and we pick a $c$ as before. Finally, if $a<t<b$ and $\varphi(t)\neq\varphi(a)$, then we choose  $c=t$. If $\varphi$ does not have any singular point, then we pick $c$ as in the case $t\leq a$ or $t\geq b$.

By Lemma \ref{lem:CMVT}, on the interval $(a,c)$, there exists some $\mu_1 \in (a,c)$ such that
 \[
(\varphi_1(c)-\varphi_1(a))\varphi'_2(\mu_1)=(\varphi_2(c)-\varphi_2(a))\varphi'_1(\mu_1).
\]
Similarly, on the interval $(c,b)$, there exists some $\mu_2 \in (c,b)$ such that
 \[
(\varphi_1(c)-\varphi_1(b))\varphi'_2(\mu_2)=(\varphi_2(c)-\varphi_2(b))\varphi'_1(\mu_2).
\]
Thus, since $\varphi(a)= \varphi(b)$ and $\varphi(c)\neq \varphi(a)$, we have
\[
\varphi'_1(\mu_1)\varphi'_2(\mu_2)=\varphi'_1(\mu_2)\varphi'_2(\mu_1)
\]
and hence the tangent lines at $\mu_1, \mu_2$ are parallel.
\end{proof}

\noindent
\begin{minipage}{0.5\textwidth}
\noindent
  \begin{remark}
  Lemma \ref{lem: self intersection} is not true for hypersurfaces in $\R^n$ when $n\geq 3$. The surface given by $\varphi \colon (t,s)\mapsto (e^{-s}(t^2-1),e^{-s}t(t^2-1),s)$ is a counterexample for $n=3$. The map $\varphi$ is not 
injective but there are no pairs of points with parallel tangent planes. The image of $\varphi$ is shown on the right.
  \end{remark}
\end{minipage}
\begin{minipage}{0.5\textwidth}
\begin{center}
  \includegraphics[width=0.7\textwidth]{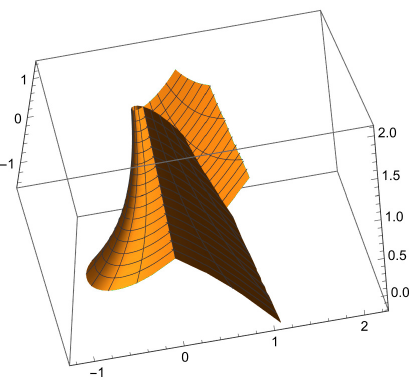}
\end{center}
\end{minipage}

Now we are able to prove the following result bounding the number of connected components of $\R^k \setminus \Gamma_\eps(A,B)$.

\begin{prop}
\label{Prop_NoCrit_NoInnerChamber}
Let $(\cA ,\eps)$ be a full-dimensional signed support of codimension $2$ with Gale dual matrix $B \in \R^{(n+3)\times 2}$. If $\bar{\xi}_{B,\eps}$ has at most one critical point, then the complement of the signed reduced A-discriminant $\Gamma_\eps(A,B)$ has at most two connected components, and each is unbounded.
\end{prop}

\begin{proof}
Recall that for $\mu \in \R$, we used the notation $\hat{\mu} = \begin{bmatrix}
    \mu \\ 1
\end{bmatrix}$. If $(\cA ,\eps)$  has a non-trivial separating hyperplane, then $\Gamma_\eps(A,B) = \emptyset$ by Theorem \ref{Thm::SepHyperplaneAdisc}. Thus, $\R^k \setminus \Gamma_\eps(A,B)$ has one connected component. If $(\cA ,\eps)$  does not have a non-trivial separating hyperplane, then from Proposition~\ref{Prop_SepHypAndKernel} follows that there exist $\mu_1 ,\mu_2 \in \R$ such that $ \sign( B \hat{\mu}_1 ) = \sign( B \hat{\mu}_2)  = \eps$. By \eqref{Eq::KapranovNormal}, $\hat{\mu}_1$ and $\hat{\mu}_2$ are normal vectors at $\bar{\xi}_{B,\eps}(\mu_1)$ and at $\bar{\xi}_{B,\eps}(\mu_2)$ respectively. If the tangent lines at $\bar{\xi}_{B,\eps}(\mu_1)$ and at $\bar{\xi}_{B,\eps}(\mu_2)$ are parallel, then $\hat{\mu}_1 = \lambda \hat{\mu}_2$ for some $\lambda \in \R \setminus \{0\}$, which implies that $\mu_1 = \mu_2$. This shows that there is no pair of points in $\Gamma_\eps(A,B)$ with parallel tangent lines.

Lemma \ref{lem: self intersection} implies that $\bar{\xi}_{B,\eps}$ is injective. Also, $\bar{\xi}_{B,\eps}$ maps an open interval of $\R$ to $\R^2$, and $\bar{\xi}_{B,\eps}(\mu)\to \infty$ as $\mu$ approaches the endpoints of the interval. Therefore, $\bar{\xi}_{B,\eps}$ is proper by Lemma \ref{lem:proper}, which implies that the complement of $\Gamma_\eps(A,B)$ has exactly two unbounded connected components by Corollary \ref{cor:unbounded components}.
\end{proof}

\subsection{Critical points and degenerate singularities}
\label{Section::DegenSingularPoints}
Let $f_c$ be an exponential sum as in \eqref{Eq::ExpSum}. A singular point $x \in \Sing(f_c)$ is called \emph{degenerate} if the Hessian matrix $\Hess_{f_c}(x)$ is not invertible. We have the following relationship between critical points of the signed reduced $A$-discriminant and degenerate singular points in the corresponding hypersurface.

\begin{lemma}
\label{Lemma::CritAndDegeneratePoints}
Let $(\cA ,\eps)$ be a full-dimensional signed support with Gale dual matrix $B \in \R^{(n+k+1) \times k }$. If $\mu^* \in \R^{k-1}$ is a critical point of $\bar{\xi}_{B,\eps}$, then for $c^* = B \begin{bmatrix}\mu^* \\  1 \end{bmatrix}$, the point $x^* = (0,\dots,0) \in \R^n$ is a degenerate singular point of $f_{c^*}$.
\end{lemma}

\begin{proof}
Since $\diag(c^*)\Big( e^{\alpha_i \cdot x^*}\Big)_{i=1,\dots,n+k+1} = B \begin{bmatrix}\mu^* \\  1 \end{bmatrix} \in \ker(\hA)$, we have that $x^*$ is a singular point of $f_{c^*}$ (cf.\eqref{Eq_FullCriticalSystem}). Thus, we only have to show that it is a degenerate singular point. Let $\psi$ denote the Horn-Kapranov Uniformization map \eqref{Eq::HornKapranov}. From \cite[Theorem 3.4, Theorem 3.5]{ForsgCusp}, it follows that  $x^*$  is a degenerate singular point if
\begin{align}
\label{Eq_RankInewJPsi}
\rk J_{\psi}( \hat{\mu}^*,x^*) \leq n+k-1,
\end{align}
where $\hat{\mu}^* =  \begin{bmatrix}\mu^* \\  1 \end{bmatrix}$.

We prove \eqref{Eq_RankInewJPsi} in two steps. First we show that
\begin{align}
\label{Eq_RankJxi}
\rk J_{\xi_{B,\eps}}(\hat{\mu}^*) = \rk J_{\bar{\xi}_{B,\eps}}(\mu^*).
\end{align}

To see this, a similar computation as in \eqref{Eq::JacobianAffineParam} shows
\[ J_{\xi_{B,\eps}}(\hat{\mu}^*) = B^\top \diag\Big( \tfrac{1}{B \hat{\mu}^*} \Big)B .\]
Thus, the first $k-1$ columns of $J_ {\xi_{B,\eps}}(\hat{\mu}^*)$ and $  J_ {\bar{\xi}_{B,\eps}}(\mu^*)$ are the same. To show that the two matrices have the same rank, it is enough to show that the last column of $J_ {\xi_{B,\eps}}(\hat{\mu}^*)$ is contained in the linear space spanned by the columns of  $J_ {\bar{\xi}_{B,\eps}}(\mu^*)$, which holds since
\[  B^\top \diag\Big( \tfrac{1}{B\hat{\mu}^*} \Big)B \hat{\mu}^* = B^\top  \mathds{1} = 0,\]
  where the last equality holds since $ B^\top \hA^\top = 0$ and the first column of $\hA^\top$ equals $\mathds{1}$. This shows \eqref{Eq_RankJxi}.

  In the second part of the proof, we show \eqref{Eq_RankInewJPsi}. Using that the diagram \eqref{Eq_ParamDiagram} commutes and the chain rule, we have
  \[ J_{B^\top \circ \Log\lvert \cdot \rvert \circ \psi} (\hat{\mu}^*,x^*) =  J_{\xi_{B,\eps} \circ \pr_k} (\hat{\mu}^*,x^*) =  J_{\xi_{B,\eps}} (\mu^*) J_{\pr_k} (\hat{\mu}^*,x^*).\]
 Using \eqref{Eq_RankJxi} and that $\rk J_{\pr_k} (\hat{\mu}^*,x^*) = k$, it follows that
  \[ \rk J_{B^\top \circ \Log\lvert \cdot \rvert \circ \psi} (\hat{\mu}^*,x^*) = \rk J_{\xi_{B,\eps}} (\hat{\mu}^*) = \rk J_{\bar{\xi}_{B,\eps}}(\mu^*)  \leq k-2,\]
  where the last inequality holds since $\mu^*$ is a critical point of $\bar{\xi}_{B,\eps}$.

  Using again the chain rule
   \[ J_{B^\top \circ \Log\lvert \cdot \rvert \circ \psi} (\hat{\mu}^*,x^*) = B^\top J_{ \Log\lvert \cdot \rvert} (\psi(\hat{\mu}^*,x^*))  J_{\psi} (\hat{\mu}^*,x^*).\]
   Note that $B^\top$ has rank $k$ and $J_{ \Log\lvert \cdot \rvert} (\psi(\hat{\mu}^*,x^*)) $ is a diagonal matrix with nonzero diagonal entries. Thus, $\rk B^\top J_{ \Log\lvert \cdot \rvert} (\psi(\hat{\mu}^*,x^*))  = k$. From Sylvester’s rank inequality follows that
   \[ \rk B^\top J_{ \Log\lvert \cdot \rvert} (\psi(\hat{\mu}^*,x^*))  + \rk  J_{\psi} (\hat{\mu}^*,x^*) - (n+k+1) \leq \rk J_{B^\top \circ \Log\lvert \cdot \rvert \circ \psi} (\hat{\mu}^*,x^*) \leq k-2,\]
  which imples \eqref{Eq_RankInewJPsi}.
\end{proof}

\begin{prop}
\label{Prop_NonDegSing_NoCrit}
Let $(\cA,\eps)$ be a full-dimensional signed support with Gale dual matrix $B$.
If for all $c \in \R_\eps^{\cA}$, all singular points of $Z(f_c)$ are non-degenerate, then $\bar{\xi}_{B,\eps}$ does not have any critical point.
\end{prop}

\begin{proof}
The statement is a direct consequence of Lemma \ref{Lemma::CritAndDegeneratePoints}.
\end{proof}

\begin{thm}
\label{Thm_NoInnerChamber} 
Let $(\cA,\eps)$ be a full-dimensional signed support of codimension $2$ with Gale dual matrix $B$. If for all $c \in \R_\eps^{\cA}$, all singular points of $Z(f_c)$ are non-degenerate, then the complement of the signed reduced $A$-discriminant $\Gamma_\eps(A,B)$ has at most two connected components, and each is unbounded.
\end{thm}

\begin{proof}
The statement follows directly from Proposition \ref{Prop_NoCrit_NoInnerChamber} and Proposition \ref{Prop_NonDegSing_NoCrit}.
\end{proof}

\subsection{Signed supports without degenerate singular points}
\label{Section::NoDegenerateSingPoints}
Now, we show that for certain signed supports $(\cA,\eps)$, the singular points of the hypersurfaces $Z(f_c)$ are non-degenerate singular for all  $c \in \R_{\eps}^{\cA}$ .

We call a pair of parallel affine hyperplanes $\mathcal{H}_{v,a},\mathcal{H}_{v,b} \subseteq \R^n$ ($a \geq b$) \emph{enclosing hyperplanes of the positive exponents} $\cA_+$ if 
\[ \cA_+ \subseteq \mathcal{H}^-_{v,a} \cap \mathcal{H}^+_{v,b} \quad \text{and} \quad  \cA_- \subseteq \R^n \setminus \big( \inte(\mathcal{H}^{-}_{v,a}) \cap \inte(\mathcal{H}^{+}_{v,b}) \big).\]
Enclosing hyperplanes $\mathcal{H}_{v,a},\mathcal{H}_{v,b}$ are \emph{strict enclosing hyperplanes} of $\cA_+$  if additionally $\inte(\mathcal{H}^{+}_{v,a}) \cap \cA_- \neq \emptyset$ and $\inte(\mathcal{H}^{-}_{v,b}) \cap \cA_- \neq \emptyset$. We define \emph{strict enclosing hyperplanes of the negative exponents} $\cA_-$ in a similar way. For an illustration, we refer to Figure \ref{FIG2}.

Our first statement concerns exponential sums in two variables.
\begin{thm}
\label{Thm_TwoEnclHyp}
Let $(\cA,\eps)$ be a full-dimensional signed support in $\R^2$ and assume that both $\cA_+$ and $\cA_-$ have a pair of strict enclosing hyperplanes. Then 
\begin{itemize}
    \item[(i)] for every $c \in \R^{\cA}_{\eps}$ and $x \in \Sing(f_c)$, the Hessian matrix $\Hess_{f_c}(x)$ has a positive and a negative eigenvalue.
    \item[(ii)] If $\cA$ consists of $5$ exponent vectors, then the complement of the signed reduced $A$-discriminant $\Gamma_\eps(A,B)$ consists of at most two connected components.
\end{itemize}
\end{thm}

\begin{proof}
Let $\mathcal{H}_{v,a}, \mathcal{H}_{v,b}$ (resp. $\mathcal{H}_{w,a'}, \mathcal{H}_{w,b'}$) enclosing hyperplanes of $\cA_+$ (resp. $\cA_-$). Using an affine change of coordinates as in Proposition \ref{Lemma::Transform}, we assume without loss of generality that $v = (1,0)^\top$, $w=(0,1)^\top$. 

For $x^* \in \Sing(f_c)$ consider the univariate exponential sums:
\begin{align*}
f_{c,x^*,v}\colon \R \mapsto \R, \quad s \mapsto f_{c,x^*,v}(s) := \sum_{i=1}^{2+k+1} c_i e^{\alpha_i\cdot( x^* + s v)} \\
f_{c,x^*,w}\colon \R \mapsto \R, \quad s \mapsto f_{c,x^*,w}(s) := \sum_{i=1}^{2+k+1} c_i e^{\alpha_i\cdot( x^* + s w)}.
\end{align*}
By construction we have:
\begin{align}
\label{Eq1:Thm_TwoEnclHyp}
 0 = f_c(x^*) = f_{c,x^*,v}(0) = f_{c,x^*,w}(0).
 \end{align}
Denoting by $\alpha_{1,i}, \alpha_{2,i}$ the first and the second coordinate of the vector $\alpha_i$, it is easy to check that
\begin{align*}
 \frac{\partial f_{c,x^*,v}}{\partial s}(s) =   \sum_{i=1}^{2+k+1} c_i \alpha_{1,i} e^{\alpha_i\cdot( x^* + s v)}, \qquad  \frac{\partial f_{c,x^*,w}}{\partial s}(s) =   \sum_{i=1}^{2+k+1} c_i \alpha_{2,i}  e^{\alpha_i\cdot( x^* + s w)},
\end{align*}
It follows that
\begin{align}
\label{Eq2:Thm_TwoEnclHyp}
 \frac{\partial f_{c,x^*,v}}{\partial s}(0) =   \frac{\partial f_{c}}{\partial x_1}(x^*) = 0 , \qquad \frac{\partial f_{c,x^*,w}}{\partial s}(0) =   \frac{\partial f_{c}}{\partial x_2}(x^*) = 0,
\end{align}
since $x^* \in Z(f_c)$ is a singular point. Combining \eqref{Eq1:Thm_TwoEnclHyp},\eqref{Eq2:Thm_TwoEnclHyp}, we have that $0$ is a root of $f_{c,x^*,v}$ (resp. $f_{c,x^*,w}$) of multiplicity at least two. 

The condition that  $\mathcal{H}_{v,a}, \mathcal{H}_{v,b}$ (resp. $\mathcal{H}_{w,a'}, \mathcal{H}_{w,b'}$) are strict enclosing hyperplanes of $\cA_+$ (resp. $\cA_-$) implies that both exponential sums have at most two sign changes in their coefficient sequence. Since Descartes' rule of signs is valid for polynomials with real exponents \cite{SimpleDescartes}, one can extend the result to exponential sums. Using Descartes' rule of signs, it follows that the multiplicity of $0$ is exactly two for both $f_{c,x^*,v}$ and $f_{c,x^*,w}$ . Furthermore,
\[f_{c,x^*,v}(s) < 0 \quad \text{and} \quad  f_{c,x^*,w}(s) > 0 \qquad \text{for all } s \neq 0.\]

So $0$ is a local maximum of $f_{c,x^*,v}$ and a local minimum of $f_{c,x^*,w}$. Therefore
\begin{align*}
  \frac{\partial^2 f_{c}}{\partial x_1^2}(x) = \frac{\partial^2 f_{c,x^*,v}}{\partial s^2}(0) <  0,  \qquad  \frac{\partial^2 f_{c}}{\partial x_2^2}(x) = \frac{\partial^2 f_{c,x^*,w}}{\partial s^2}(0) > 0,
\end{align*}
which implies that
\[ \det(\Hess_{f_c}(x^*)) =   \frac{\partial^2 f_{c}}{\partial x_1^2}(x^*)   \frac{\partial^2 f_{c}}{\partial x_2^2}(x^*) - \Big( \frac{\partial^2 f_{c}}{\partial x_1 \partial x_2}(x^*) \Big)^2 < 0.\]
Thus, $\Hess_{f_c}(x^*)$ is invertible and must have a negative and a positive eigenvalue.

If $\cA$ contains $5$ exponent vectors, then the codimension of $\cA$ is $2$. Since all singular points of $Z(f_c)$ are non-degenerate for all $c \in \R^{\cA}_\eps$ by (i), part (ii) follows from Theorem \ref{Thm_NoInnerChamber}.
\end{proof}

\begin{figure}[t]
\centering
\includegraphics[scale=0.6]{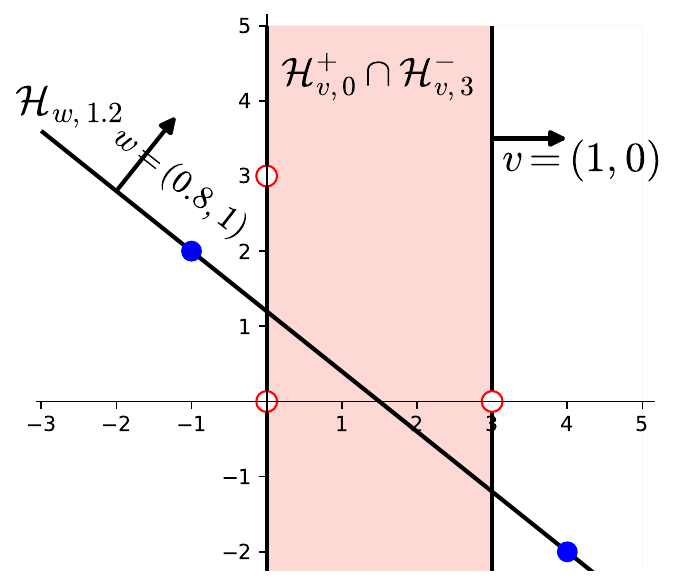}
\caption{{\small An illustration of strict enclosing hyperplanes. Consider the signed exponent vectors $\cA_+ = \left\{ \begin{bmatrix} 0 \\ 0 \end{bmatrix},\begin{bmatrix} 3 \\ 0 \end{bmatrix}, \begin{bmatrix} 0 \\ 3 \end{bmatrix} \right\}$  (depicted by red circles) and  $\cA_- = \left\{\begin{bmatrix} -1 \\ 2  \end{bmatrix},\begin{bmatrix} 4 \\ -2  \end{bmatrix} \right\}$ (depicted by blue dots). The hyperplanes   $\mathcal{H}_{v,3},\mathcal{H}_{v,0}$ with $v = (1,0)$  are strict enclosing hyperplanes of $\cA_+$. The negative exponent vectors $\cA_-$ also have a pair of strict enclosing hyperplanes given by $\mathcal{H}_{w,1.2},\mathcal{H}_{w,1.2}$ with  $w = (0.8,1)$.}}\label{FIG2}
\end{figure}

The next condition on the signed support that precludes the existence of degenerate singular points is valid for every number of variables $n$. Specifically, we require that there is only one exponent vector with negative sign.

\begin{thm}
\label{Thm::OneNegative}
Let $(\cA,\eps)$ be a full-dimensional signed support with Gale dual $B$ such that $\# \cA_- = 1$. Then we have
\begin{itemize}
    \item[(i)] for all $c \in \R_\eps^{\cA}$ and $x \in \Sing(f_c)$ the Hessian matrix $\Hess_{f_c}(x)$ has only positive eigenvalues.
    \item[(ii)] If $(\cA,\eps)$ has codimension $2$, then the complement of the signed reduced $A$-discriminant $\Gamma_{\eps}(A,B)$ consists of at most two connected components.
\end{itemize}
\end{thm}

\begin{proof}
Write $\cA_+ = \{ \alpha_1, \dots , \alpha_{n+k} \}$ and  $\cA_- = \{ \alpha_{n+k+1} \}$. Using Proposition \ref{Lemma::Transform}, we assume without loss of generality that $\alpha_{n+k+1} = 0$. Under this assumption, the Hessian of $f_c$ at $x \in \Sing(f_c)$ is given by
\begin{align}
\label{Eq_OneNegProof}
 \Hess_{f_c}(x)  = \sum_{i=1}^{n+k} (e^{\alpha_i \cdot x}c_i) \alpha_i \cdot \alpha_{i}^\top = \tilde{A} \diag \big( (e^{\alpha_i \cdot x}c_i)_{i=1,\dots,n+k}) \big) \tilde{A}^\top,
 \end{align}
where
\[ \tilde{A} = \begin{bmatrix} \alpha_1 & \dots & \alpha_{n+k} \end{bmatrix} \in \R^{n \times (n+k)}.\]
Since the affine hull of $\alpha_1, \dots , \alpha_{n+k+1}$ has dimension $n$ and $\alpha_{n+k+1} = 0$, it follows that $\rk(\tilde{A}) = n$.

Since $e^{\alpha_i \cdot x}c_i$ is positive for $i =1,\dots,n+k$, their square root is a real number. This gives
\[
\Hess_{f_c}(x)  = \big( \tilde{A} \diag \big( (\sqrt{e^{\alpha_i \cdot x}c_i})_{i=1,\dots,n+k}) \big)\big( \tilde{A}\diag \big( (\sqrt{e^{\alpha_i \cdot x}c_i})_{i=1,\dots,n+k}) \big)  \big)^\top
\]
and as $\tilde{A}$ has full rank
\[ \rk\big( \tilde{A} \diag \big( (e^{\alpha_i \cdot x}c_i)_{i=1,\dots,n+k}) \big) \tilde{A}^\top \big) = \rk( \tilde{A} \diag \big( (\sqrt{e^{\alpha_i \cdot x}c_i})_{i=1,\dots,n+k}) \big) )  = \rk \tilde{A} = n.\]
Thus, $ \Hess_{f_c}(x)$ is positive semi-definite and of full rank, which implies that all of its eigenvalues are positive. This shows (i).

Since  all singular points of $Z(f_c)$ are non-degenerate for all $c \in \R^{\cA}_\eps$, part (ii) follows from Theorem \ref{Thm_NoInnerChamber}.
\end{proof}

Our final condition on the signed support, precluding the existence of degenerate singular points in the hypersurfaces $Z(f_c)$, requires the positive and negative exponent vectors to be separated by a simplex, as follows. We recall the definition of the negative vertex cone of a simplex from \cite[Section 4]{DescartesHypPlane}. For an $n$-simplex $P \subseteq \R^n$ with vertices $\mu_0, \dots, \mu_n$, the \emph{negative vertex cone} at vertex $\mu_k$ equals
\[P^{-,k} := \mu_k + \Cone\big( \mu_k - \mu_0, \dots , \mu_k - \mu_n \big).\]
We write $P^-$ for the union of $P^{-,0}, \dots , P^{-n}$. We refer to Figure \ref{FIG3} for such a simplex and its negative vertex cones in the plane.

\begin{thm}
\label{Thm::Simplex}
Let $P \subseteq \R^{n}$ be an $n$-simplex and let $(\cA,\eps)$ be a signed support in $\R^n$ with Gale dual matrix $B$ such that $\cA_+ \subseteq P$, $\cA_- \subseteq P^-$, and $\cA \cap \inte \big( P \cup P^-\big) \neq \emptyset$. Then 
\begin{itemize}
    \item[(i)] for all $c \in \R^{\cA}_\eps$ and all singular points $x \in \Sing(f_c)$ the eigenvalues of $\Hess_{f_c}(x)$ are negative.
    \item[(ii)] If $\cA$ has codimension $2$, then the complement of the signed reduced A-discriminant $\Gamma_{\eps}(A,B)$ consists of at most two connected components.
\end{itemize}
\end{thm}

\begin{proof}
Note that the negative vertex cones are preserved under affine transformation of $P$, see e.g. \cite[Lemma 4.5]{DescartesHypPlane}. Thus by Proposition \ref{Lemma::Transform}, we can assume without loss of generality that $P = \Conv( 0, e_1, \dots , e_n)$.

Denote $\Exp\colon \R^n \to \R^n_{>0}$ and $\Log \colon \R^n_{>0} \to \R^n$ the coordinate-wise exponential and logarithm maps. From \cite[Theorem 7]{Maranas_Floudas}, it follows that the Hessian of the function
\[ f_c \circ \Log \colon \R^n_{>0}\to \R, \quad  y \mapsto \sum_{i=1}^{n+k+1}c_i y^{\alpha_i}\]
is negative definite for all $y \in \R^n_{>0}$. If $x \in \R^n$ is a singular point of $f_c$, then from \cite[Corollary 1]{skorski2019chain} it follows that
\[ \Hess_{f_c}(x) = \Hess_{f_c \circ \Log \circ \Exp}(x) =  \big( J_{\Exp} (x) \big)^\top \Hess_{f_c \circ \Log} (\Exp(x)) J_{\Exp}(x). \]
Thus, all the eigenvalues of $\Hess_{f_c}(x)$ are negative by Sylvester's law of inertia \cite[Chapter 7]{meyer04}. In particular, all singular points of $Z(f_c)$ are non-degenerate for all $c \in \R^{\cA}_\eps$.  Part (ii) follows from Theorem \ref{Thm_NoInnerChamber}.
\end{proof}

\begin{figure}[t]
\centering
\includegraphics[scale=0.7]{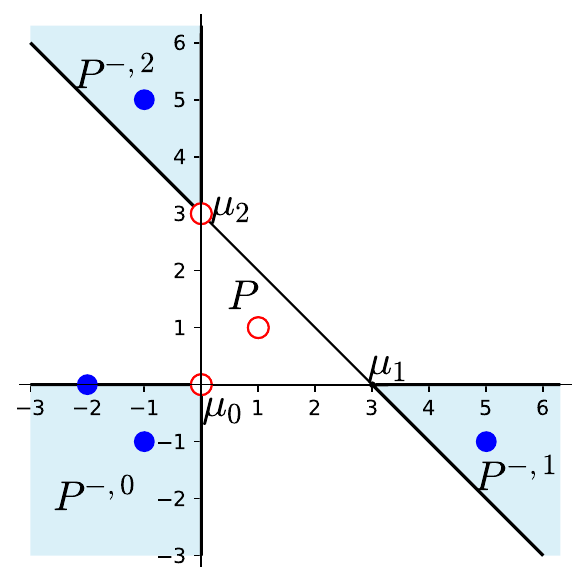}
\caption{{\small A simplex $P = \Conv( (0,0),(3,0),(0,3))$ separating the signed exponent vectors $\cA_+ = \left\{ \begin{bmatrix} 0 \\ 0 \end{bmatrix},\begin{bmatrix} 0 \\ 3 \end{bmatrix}, \begin{bmatrix} 1 \\ 1 \end{bmatrix} \right\}$  (marked by red circles) and  $\cA_- = \left\{\begin{bmatrix} -1 \\ 5  \end{bmatrix},\begin{bmatrix} 5 \\ -1  \end{bmatrix},\begin{bmatrix} -1 \\ -1  \end{bmatrix},\begin{bmatrix} -2 \\ 0  \end{bmatrix} \right\}$ (marked by blue dots). }}\label{FIG3}
\end{figure}

Using Theorem \ref{Thm::OneNegative} and Theorem \ref{Thm::Simplex}, we give conditions on the signed support $(\cA,\eps)$ such that all possible isotopy types of $Z(f_c), \, c \in \R^\cA_\eps$ are given by some signed tropical hypersurfaces (cf. Theorem \ref{Thm:Viro}).

\begin{cor}
   \label{Thm::ViroAll}
   Let $(\cA,\eps)$ be a full-dimensional signed support of codimension $2$ with Gale dual $B$ such that either $\# \cA_- = 1$ or $\cA_+$ and $\cA_-$ are separated by a  simplex as in Theorem \ref{Thm::Simplex}. If for each proper face $F  \subsetneq \Conv(\cA)$ the restricted signed support $(\cA_F,\eps_F)$ has a non-trivial separating hyperplane, then for each smooth hypersurface $Z(f_c)$ with $c \in \R_\eps^\cA$ there exists $h \in \R^\cA$ such that the signed tropical hypersurface $\Trop_\eps(\cA,h)$ and $Z(f_c)$ have the same isotopy type.
\end{cor}

\begin{proof}
    In both cases, the complement of $\Gamma_\eps(A,B)$ has at most two connected components by Theorem \ref{Thm::OneNegative} or Theorem \ref{Thm::Simplex}. Since $(\cA_{F},\eps_F)$ has a non-trivial separating hyperplane for every proper face $F \subsetneq \Conv(\cA)$, we have  $\nabla_{\cA_F,\eps_F} = \emptyset$ by Theorem \ref{Thm::SepHyperplaneAdisc}. From Theorem \ref{Thm::RojasRusek_Adiscriminant} follows that the hypersurfaces $Z(f_c)$ with $c \in \R^\cA_\eps$  have at most two different isotopy types.

    First, we focus on the case $\# \cA_- = 1$. Assume with out loss of generality that $\cA_- = \{ \alpha_{n+3}\}$. If $\alpha_{n+3}$ is contained in the boundary of $\Conv(\cA)$, then there exist a hyperplane $\mathcal{H}_{v,a} \subseteq \R^n$ such that $\alpha_{n+3} \in \mathcal{H}_{v,a}$ and $\cA \subseteq \mathcal{H}_{v,a}^+$ (cf. \cite[Corollary 2.5]{JoswigTheobald_book}). Thus $(\cA,\eps)$ has a non-trivial separating hyperplane, which implies that  all $Z(f_c)$ with $c \in \R^{\cA}_\eps$ have the same isotopy type by Theorem~\ref{Thm::SepHyperplane}. 
    
   If $\alpha_{n+3} \in \inte( \Conv(\cA))$, then choose a generic $h \in \R^{\cA} \cong  \R^{n+3}$ such that $h_{n+3} > h_i$ for $i =1, \dots , n+2$. By construction, we have  $\Trop_\eps(\cA,h) \neq \emptyset$ and $\Trop_\eps(\cA,-h) = \emptyset$. By Theorem \ref{Thm:Viro}, there exist $c_1,c_2 \in \R^\cA_\eps$ such that $Z(f_{c_1})$ and $Z(f_{c_2})$ are isotopic to  $\Trop_\eps(\cA,h)$ and to $\Trop_\eps(\cA,-h)$ respectively. Since the number of possible isotopy types is at most two, it follows that the possible isotopy types are given by $\Trop_\eps(\cA,h)$ and  $\Trop_\eps(\cA,-h)$.

    Next, we consider the case when $\cA_+$ and $\cA_-$ are separated by a simplex $P \subseteq \R^n$. If one of the negative simplex cones $P^{-,0}, \dots , P^{-,n}$ does not contain any positive exponent vector, then $(\cA,\eps)$ has a non-trivial separating hyperplane and all $Z(f_c)$ with $c \in \R^{\cA}_\eps$ have the same isotopy type by Theorem \ref{Thm::SepHyperplane}. If $P^{-,i} \cap \cA_+ \neq \emptyset$ for each $i=0,\dots,n$, then a similar argument as above shows that there exists two signed tropical hypersurfaces which are not isotopic to each other. This concludes the proof.
\end{proof}

\section{Bivariate $5$-nomials}
\label{Section::5nomials}
For bivariate $5$-nomials, the signed reduced $A$-discriminant has at most $2$ critical points by Lemma \ref{Lemma:NumCritPoints}. If there is only one critical point, then $\R^k \setminus \Gamma_\eps(A,B)$ has a simple structure, it has at most two connected components, which are both unbounded (cf. Proposition \ref{Prop_NoCrit_NoInnerChamber}). In this section, we give a complete description of the geometry of the signed support of a bivariate $5$-nomial whose signed reduced $A$-discriminant has two critical points. In our experiments, if the signed reduced $A$-discriminant had two critical points,  then its complement had a bounded chamber. We conjecture that this is always true, however we do not have a proof of this statement nor a counter example.

\begin{conj}
\label{Lemma:2CritInnerChamber}
Let $(\cA,\eps)$ be the signed support of a bivariate $5$-nomial and 
    let 
    \[\bar{\xi}_{B,\eps}\colon \{ \mu \in \R \mid  \sign(B \begin{bmatrix}\mu \\  1 \end{bmatrix}) = \eps \} \to \R^2\]
    be the parametrization map of $\Gamma_\eps(A,B)$ as defined in \eqref{Eq_AffineSignedRedParam}. If $\bar{\xi}_{B,\eps}$ has two critical points, then the complement of $\Gamma_\eps(A,B)$ has a bounded connected component.
\end{conj}

Given a $2$-simplex $P = \Conv(\mu_0, \mu_1,\mu_2)$, denote  by $\mathcal{H}_{v_0,d_0}, \mathcal{H}_{v_1,d_1},  \mathcal{H}_{v_2,d_2}$ the supporting hyperplanes of the facets of $P$. We choose these hyperplanes such that
\[ P =  \mathcal{H}^+_{v_0,d_0}  \cap \mathcal{H}^+_{v_1,d_1}  \cap \mathcal{H}^+_{v_2,d_2}  \]
and $\mu_i \notin \mathcal{H}_{v_i,d_i}$ for each $i = 0,1,2$. The complement of the union of the hyperplanes $\mathcal{H}_{v_0,d_0} , \mathcal{H}_{v_1,d_1} ,\mathcal{H}_{v_2,d_2} $ has $7$ chambers. One of these chambers is the simplex $P$. Three other chambers are the negative vertex cones $P^{-,0}, P^{-,1},P^{-,2}$, as introduced in Section \ref{Section::NoDegenerateSingPoints}. For these chambers we have
\begin{align}
\label{Eq::ChamberNeg}
 P^{-,i} = \bigcap_{j = 0, j \neq i}^2 \mathcal{H}^-_{v_j,d_j} \cap \mathcal{H}^+_{v_i,d_i}, \quad \text{ for } i = 0,1,2.
 \end{align}
The three other chambers in the hyperplane arrangment can be written as
\begin{align}
\label{Eq::ChamberPos}
 P^{+,i} = \bigcap_{j = 0, j \neq i}^2 \mathcal{H}^+_{v_j,d_j} \cap \mathcal{H}^-_{v_i,d_i}, \quad \text{ for } i = 0,1,2.
  \end{align}
For $i \neq j \in \{0,1,2\}$, we define the subset of $P^{+,i}$
\begin{align}
\label{Eq::ChamberPosSub}
 P^{+,i,j} :=  \bigcap_{j = 0, j \neq i}^2 \mathcal{H}^+_{v_j,d_j} \cap \mathcal{H}^-_{v_i,d_i} \cap \mathcal{H}^-_{v_j,D_j},
  \end{align}
where $D_j := v_j \cdot \mu_j > d_j$. For an illustration of these chambers, we refer to Figure \ref{FIG4}.

In the following lemmata, we focus on the special case when $P$ is the standard $2$-simplex $\Delta_2 = \Conv((0,0)^\top,(1,0)^\top,(0,1)^\top)$ and 
\begin{align}
\label{Eq::23_AMatrix}
 \cA_+ = \left\{ \begin{bmatrix} 0 \\ 0 \end{bmatrix},  \begin{bmatrix} 1 \\ 0 \end{bmatrix},\begin{bmatrix} 0 \\ 1 \end{bmatrix}\right\}, \quad \cA_- = \left\{  \begin{bmatrix} x_1 \\ y_1 \end{bmatrix}, \begin{bmatrix} x_2 \\ y_2 \end{bmatrix}\right\}.
 \end{align}
 Afterward we extend the results to the general case in Theorem \ref{Thm::5nomials}. Choose a Gale dual matrix corresponding to the exponent vectors in \eqref{Eq::23_AMatrix}  as
\begin{align}
\label{Eq::23_BMatrix}
B = \begin{bmatrix} 1-x_1-y_1 & 1-x_2-y_2 \\
x_1 & x_2\\
y_1 & y_2\\
-1 & 0 \\
0 & -1
\end{bmatrix}.
\end{align}
With that choice of the Gale dual matrix, the polynomial from \eqref{Eq::CritPoly} is a quadratic polynomial  $q_B(t) := a t^2 + b t + c$, where
\begin{equation}
    \label{Sec5_qbCoeff}
\begin{aligned}
a &:= -x_1 y_1 \left(1-x_1 -y_1 \right), \\
b &:= -x_1^{2} y_2^{2}+2 x_1 x_2 y_1 y_2 -x_2^{2} y_1^{2}+x_1^{2} y_2 +y_2^{2} x_1 +x_2^{2} y_1 +y_1^{2} x_2 -x_1 y_2 -x_2 y_1, \\
c &:=  -x_2 y_2 \left(1-x_2 -y_2\right).
\end{aligned}
\end{equation}

A point $t \in \R$ is a critical point of $\bar{\xi}_{B,\eps}$, $\eps = (1,1,1,-1-1)$ if and only if $q_B(t) = 0$ and it satisfies the inequalities:
\begin{equation}
\label{Eq::ConstOn_t}
\begin{aligned}
(1-x_1-y_1)t + 1-x_2-y_2 >0, \quad x_1 t + x_2 >0, \quad y_1 t + y_2 >0, \quad t >0.
\end{aligned}
\end{equation}

\begin{lemma}
\label{Lemma:5nom:1}
Let $\cA_+,\cA_-$ be the set of exponent vectors as defined in \eqref{Eq::23_AMatrix} and let $B$ the corresponding Gale dual matrix from \eqref{Eq::23_BMatrix}. If
\begin{itemize}
    \item[(i)] $\alpha_4 \in \Delta_2$ and $\alpha_5  \in \Delta_2^{+,i}$ for some $i \in \{ 0,1,2 \}$, or 
    \item[(ii)] $\alpha_4 \in \Delta_2^{-,i}$ and $\alpha_5  \in \Delta_2^{+,i}$ for some $i \in \{0,1,2\}$,
\end{itemize} then $\bar{\xi}_{B,\eps}$ has at most one critical point.
\end{lemma}

\begin{proof}
Let $a,c$ denote the coefficients of $q_B$ as in \eqref{Sec5_qbCoeff}. Both in case (i) and (ii), we have $a \leq 0, \, c \geq 0$, which implies that $q_B$ has at most one sign change in its coefficient sequence. From Descartes' rule of signs, it follows that $q_B$ has at most $1$ positive real root. By \eqref{Eq::ConstOn_t}, every critical point of $\bar{\xi}_{B,\eps}$ is a positive root of $q_B$. Therefore, $\bar{\xi}_{B,\eps}$ has at most one critical point.
\end{proof}

\begin{lemma}
\label{Lemma:5nom:2}
 Let $\cA_+,\cA_-$ be the set of exponent vectors as defined in \eqref{Eq::23_AMatrix} and let $B$ the corresponding Gale dual matrix from \eqref{Eq::23_BMatrix}. If $\alpha_4 \in \inte(\Delta_2)$ and $\alpha_5 \in \inte(\Delta_2^{-,0})$, then $\bar{\xi}_{B,\eps}$ has one critical point.
\end{lemma}

\begin{proof}
The inequalities in \eqref{Eq::ConstOn_t} are equivalent to $M := \max\{ \tfrac{-x_2}{x_1}, \tfrac{-y_2}{y_1}\} < t$. Note that $M > 0$, The number of critical points of $\bar{\xi}_{B,\eps}$ is the same as the number of roots of $q_B$ in the interval $(M,\infty)$.

    Let $a,c$ denote the coefficients of $q_B$ as in \eqref{Sec5_qbCoeff}. Under the assumption of the lemma, we have $a < 0$ and $c < 0$. Thus, $q_B$ has $0$ or $2$ sign changes in its coefficient sequence. By Descartes' rule of signs $q_B $ has at most two positive roots. Moreover, if $q_B(M) > 0$, then $q_B$ has exactly one root in the interval $(M,\infty)$.

    Evaluating $q_B$ at $\tfrac{-x_2}{x_1}$ or at $\tfrac{-y_2}{y_1}$, depending which one is larger, we used the \texttt{Maple}  function \texttt{IsEmpty} \cite{maple} and the \texttt{Mathematica} function \texttt{Reduce} \cite{Mathematica}, to verify that $q_B(M) > 0$. Thus, $q_B$ has exactly one root in the interval $(M,\infty)$, which concludes the proof.
\end{proof}

\begin{lemma}
\label{Lemma:5nom:3}
  Let $\cA_+,\cA_-$ be the set of exponent vectors as defined in \eqref{Eq::23_AMatrix}, let $B$ the corresponding Gale dual matrix from \eqref{Eq::23_BMatrix} and let $a,b,c$ defined in \eqref{Sec5_qbCoeff}. Assume $\alpha_4 \in \inte(\Delta_n^{+,1})$ and $\alpha_5 \in \inte(\Delta_n^{+,2})$. The map $\bar{\xi}_{B,\eps}$ has two critical points if and only if
   $\alpha_4 \in \inte(P^{+,1,2})$ and $\alpha_5 \in \inte(P^{+,2,1})$ and the coordinates of $\alpha_4$ and $\alpha_5$ satisfy the following inequalities:
\begin{equation}
\label{Eq::Thm5nomials}
\begin{aligned}
b^2 -4ac &> 0\\
b^2 - 4 ac &< (2x_2y_1(1-x_1-y_1)+b)^2, \quad 0 < 2x_2y_1(1-x_1-y_1)+b \\
b^2 - 4 ac &< (2x_1y_2(1-x_1-y_1)+b)^2, \quad 0 > 2x_1y_2(1-x_1-y_1)+b.
\end{aligned}
\end{equation}
\end{lemma}

\begin{proof}
If $\relint(\Conv(\{\alpha_4,\alpha_5\}) \cap \relint(\Delta_2) = \emptyset$, then $\cA_+$ and $\cA_-$ can be separated by an affine hyperplane \cite[Section 2.2, Theorem 2]{grunbaum2003convex} and therefore $\Gamma_\eps(A,B) = \emptyset$ by Theorem \ref{Thm::SepHyperplaneAdisc}. In particular, $\Gamma_\eps(A,B)$ does not have any critical point.

If $\relint(P) \cap \relint(Q) \neq \emptyset$, then there exists a pair of strict enclosing hyperplanes of $\cA_-$,which are parallel to the affine hull of $\alpha_4$ and $\alpha_5$. If additionally $\alpha_5 \in  \inte(\Delta_n^{+,2}) \setminus  \inte(\Delta_n^{+,2,1})$, then by perturbing the hyperplanes $\mathcal{H}_{e_1,0},\mathcal{H}_{e_1,1}$ we get strict enclosing hyperplanes of $\cA_+$. Thus, from  Proposition \ref{Prop_NonDegSing_NoCrit} and Theorem \ref{Thm_TwoEnclHyp} it follows that $\bar{\xi}_{B,\eps}$ does not have any critical point.

A similar argument shows that $\bar{\xi}_{B,\eps}$ does not have critical points if  $\alpha_5 \in  \inte(\Delta_n^{+,1}) \setminus  \inte(\Delta_n^{+,1,2})$. This shows that to get a critical point of $\bar{\xi}_{B,\eps}$ the negative exponent vectors must satisfy $\alpha_4 \in \inte(\Delta_n^{+,1,2})$ and  $\alpha_5 \in \inte(\Delta_n^{+,2,1})$, which is equivalent to
\begin{align*}
x_1 < 0, \quad 0 < y_1< 1, \quad 1-x_1-y_1 >0, \\
0 < x_2 < 1, \quad y_2 < 0, \quad 1-x_2-y_2 >0.
\end{align*}
Assuming these inequalities are satisfied, the inequality $y_1 t + y_2 >0$ in \eqref{Eq::ConstOn_t} implies $t > \tfrac{-y_2}{y_1} > 0$, and $t > 0$ implies $(1-x_1-y_1) t + 1-x_2-y_2 >0$. Thus, the first and the last inequalities in \eqref{Eq::ConstOn_t} are redundant.

The roots of $q_B$ are given by
\[ t_{1} = \tfrac{-b + \sqrt{b^2-4ac}}{2a}, \qquad t_{2} = \tfrac{-b - \sqrt{b^2-4ac}}{2a}. \]
An easy computation shows that $t_1 \neq t_2$ and both satisfy the second and the third inequality in \eqref{Eq::ConstOn_t} if and only if
\begin{align*}
b^2 -4ac &> 0\\
b^2 - 4 ac &< (2x_2y_1(1-x_1-y_1)+b)^2, \quad 0 < 2x_2y_1(1-x_1-y_1)+b \\
b^2 - 4 ac &< (2x_1y_2(1-x_1-y_1)+b)^2, \quad 0 > 2x_1y_2(1-x_1-y_1)+b.
\end{align*}
This concludes the proof.
\end{proof}

Using Lemma \ref{Lemma:5nom:1}, Lemma \ref{Lemma:5nom:2}, and Lemma \ref{Lemma:5nom:3}, we characterize the signed support of a bivariate $5$-nomial such that signed reduced $A$-discriminant has two critical points.

\begin{thm}
\label{Thm::5nomials}
Let $(\cA,\eps)$ be the full-dimensional signed support of a bivariate $5$-nomial with Gale dual matrix $B \in \R^{5\times 2}$.
 The map $\bar{\xi}_{B,\eps}$ has two critical points only if $\# \cA_+ = 3, \, \# \cA_- = 2$ and $\dim \Conv(\cA_+) =2$, or  $\# \cA_+ = 2, \, \# \cA_- = 3$ and $\dim \Conv(\cA_-) =2$.

Assume that $\cA_+ = \{ \alpha_1 , \alpha_2 , \alpha_3\}, \, \cA_- = \{ \alpha_4 , \alpha_5\}$ and $P = \Conv(\cA_+)$ has dimension $2$. Let $M \in \R^{2 \times 2}$ be an invertible matrix  such that $M (\alpha_2-\alpha_1) = e_1, \, M (\alpha_3-\alpha_1) = e_2$ and $v = -M\alpha_1$. Denote $(x_1,y_1)^\top = M \alpha_4 + v$, $(x_2,y_2)^\top = M \alpha_5 + v$ and $a,b,c$ the expressions in $x_1,y_1,x_2,y_2$ from \eqref{Sec5_qbCoeff}.  The map $\bar{\xi}_{B,\eps}$ has two critical points if and only if $\alpha_4 \in \inte(P^{+,j,i})$ and $\alpha_5 \in \inte(P^{+,i,j})$ for $i \neq j \in \{0,1,2\}$ and the coordinates of $\alpha_4$ and $\alpha_5$ satisfy the following inequalities:
\begin{equation}
\label{Eq::Thm5nomials}
\begin{aligned}
b^2 -4ac &> 0\\
b^2 - 4 ac &< (2x_2y_1(1-x_1-y_1)+b)^2, \quad 0 < 2x_2y_1(1-x_1-y_1)+b \\
b^2 - 4 ac &< (2x_1y_2(1-x_1-y_1)+b)^2, \quad 0 > 2x_1y_2(1-x_1-y_1)+b.
\end{aligned}
\end{equation}
\end{thm}

\begin{proof}
If all the exponent vectors are positive (resp. negative), then $\Gamma_\eps(A,B) = \emptyset$. If $\#\cA_+ = 1$ or $\#\cA_- = 1$, then  $\bar{\xi}_{B,\eps}$ does not have any critical point by  Theorem \ref{Thm::OneNegative} and Proposition \ref{Prop_NonDegSing_NoCrit}. Thus, in order to have a critical point of  $\bar{\xi}_{B,\eps}$ one needs $\# \cA_+ \geq 2, \, \# \cA_- \geq 2$.

Write $P = \Conv(  \cA_+ )$ and $Q = \Conv( \cA_-)$ and assume $\dim P = \dim Q = 1$.  Since $\dim \Conv( P \cup Q ) = \dim \Conv( \cA ) =2$, the intersection $P\cap Q$ is either empty or a point. If $\relint(P) \cap \relint(Q) = \emptyset$, then $P$ and $Q$ can be separated by an affine hyperplane \cite[Section 2.2, Theorem 2]{grunbaum2003convex} and therefore $\Gamma_\eps(A,B) = \emptyset$ by Theorem \ref{Thm::SepHyperplaneAdisc}. If $\relint(P) \cap \relint(Q) \neq \emptyset$, then the affine hull of $P$ (resp. $Q$) is a strict enclosing hyperplane of $\cA_+$ (resp. $\cA_-$). In that case, we use Theorem \ref{Thm_TwoEnclHyp} and Proposition \ref{Prop_NonDegSing_NoCrit} to conclude that $\bar{\xi}_{B,\eps}$ does not have any critical point. This shows the first part of the theorem.

\medskip
In the rest of the proof, we assume  $\cA_+ = \{ \alpha_1 , \alpha_2 , \alpha_3\}, \, \cA_- = \{ \alpha_4 , \alpha_5\}$ and $\dim P = 2$. Choose the order of $\alpha_1, \alpha_2, \alpha_3$ so that ${\displaystyle \det\begin{bmatrix}
    1 & 1 & 1\\
    \alpha_1 & \alpha_2 & \alpha_3
\end{bmatrix}>0
}$. Then there exists an invertible matrix $M \in \R^{2 \times 2}$ with positive determinant such that $M (\alpha_2-\alpha_1) = e_1$ and $M (\alpha_3-\alpha_1) = e_2$. Let $v := -M\alpha_1$. 

By construction, the affine linear map
\[ L\colon \, \R^2 \to \R^2, \quad a \mapsto Ma +v,\]
satisfies $L(\alpha_1) = 0, \, L(\alpha_2) = e_1, \, L(\alpha_3) = e_2$ and $L(P) = \Delta_2$.  By Proposition  \ref{Lemma::Transform}, we assume without loss of generality that
\begin{align*}
 \alpha_1 = \begin{bmatrix} 0 \\ 0 \end{bmatrix}, \alpha_2 = \begin{bmatrix} 1 \\ 0 \end{bmatrix},\alpha_3 = \begin{bmatrix} 0 \\ 1 \end{bmatrix},\alpha_4 = \begin{bmatrix} x_1 \\ y_1 \end{bmatrix},\alpha_5 = \begin{bmatrix} x_2 \\ y_2 \end{bmatrix},
 \end{align*}
and choose the Gale dual matrix $B$ as in \eqref{Eq::23_BMatrix}.

 If $\alpha_4, \alpha_5$ are separated from $\Delta_2$ by an affine hyperplane, then $\Gamma_\eps(A,B) = \emptyset$ by Theorem \ref{Thm::SepHyperplaneAdisc}. If $\alpha_4, \alpha_5 \in \inte(\Delta_2)$, then $\bar{\xi}_{B,\eps}$ does not have any critical point by Theorem \ref{Thm::Simplex} and Proposition \ref{Prop_NonDegSing_NoCrit}. If $\alpha_4$ (resp. $\alpha_5 $) lies on a supporting hyperplane of a facet of $P$, then $a = 0$ (resp. $c = 0$). From this follows that $q_B$ has at most two monomial terms, so $q_B$ has at most one positive root. Thus, $\bar{\xi}_{B,\eps}$ has at most one critical point.
 
 In the following, we investigate the remaining cases:
\begin{itemize}
\item[(I)] $\alpha_4 \in \inte(\Delta_2)$ and $\alpha_5  \in \inte(\Delta_2^{+,i})$ for some $i \in \{ 0,1,2 \}$.
\item[(II)] $\alpha_4 \in \inte(\Delta_2^{-,i})$ and $\alpha_5  \in \inte(\Delta_2^{+,i})$ for some $i \in \{0,1,2\}$.
\item[(III)] $\alpha_4 \in \inte(\Delta_2)$ and $\alpha_5 \in \inte(\Delta_2^{-,i})$ for some $i \in \{ 0,1,2 \}$.
\item[(IV)] $\alpha_4 \in \inte(\Delta_2^{+,i})$ and $\alpha_5  \in \inte(\Delta_2^{+,j})$ for some $i\neq j \in \{0,1,2\}$.
\end{itemize}
For (I) and (II), Lemma \ref{Lemma:5nom:1} implies that $\bar{\xi}_{B,\eps}$ has at most one critical point. In case (III), by rotating the exponent vectors we assume without loss of generality that $\alpha_5 \in \inte(\Delta_2^{-,0})$. Now, it follows from Lemma \ref{Lemma:5nom:2} that $\bar{\xi}_{B,\eps}$ has one critical point. Under the assumption in (IV), we rotate again the exponent vectors to achieve $\alpha_4 \in \inte(\Delta_2^{+,1})$ and $\alpha_5  \in \inte(\Delta_2^{+,2})$. This rotation does not change the number of critical points of $\bar{\xi}_{B,\eps}$ by Corollary \ref{Cor:AdiscNotChanged}. We use Lemma \ref{Lemma:5nom:3} to conclude that $\bar{\xi}_{B,\eps}$ has two critical points if and only if the inequalities in \eqref{Eq::Thm5nomials} are satisfied.
\end{proof}

\begin{remark}
\label{Remark::5nomials}
Consider the signed support 
\begin{align*}
\cA_+ =\left\{
 \begin{bmatrix} 0 \\ 0 \end{bmatrix}, \begin{bmatrix} 1 \\ 0 \end{bmatrix}, \begin{bmatrix} 0 \\ 1 \end{bmatrix} \right\}, \quad \cA_- = \left\{  \begin{bmatrix} x_1 \\ y_1 \end{bmatrix}, \begin{bmatrix} x_2 \\ y_2 \end{bmatrix} \right\}.
 \end{align*}
 Using the \texttt{Mathematica} function \texttt{Reduce} \cite{Mathematica}, we verified that for every fixed $(x_1,y_1) \in \inte(\Delta_2^{+,1,2})$ and $0 < x_2 < 1$, there always exists $y_2 < 0$ satisfying the inequalities in \eqref{Eq::Thm5nomials}. In other words, for given $(x_1,y_1) \in \inte(\Delta_2^{+,1,2})$ there exists a $(x_2,y_2) \in \inte(\Delta_2^{+,2,1})$ such that $\bar{\xi}_{B,\eps}$ has two critical points.
 
 In Figure \ref{FIG4}, we depicted the region of  $(x_2,y_2)$'s satisfying the inequalities in \eqref{Eq::Thm5nomials} for $(x_1,y_1) = (-0.1,0.3)$.
\end{remark}

\begin{figure}[t]
\centering
\includegraphics[scale=0.7]{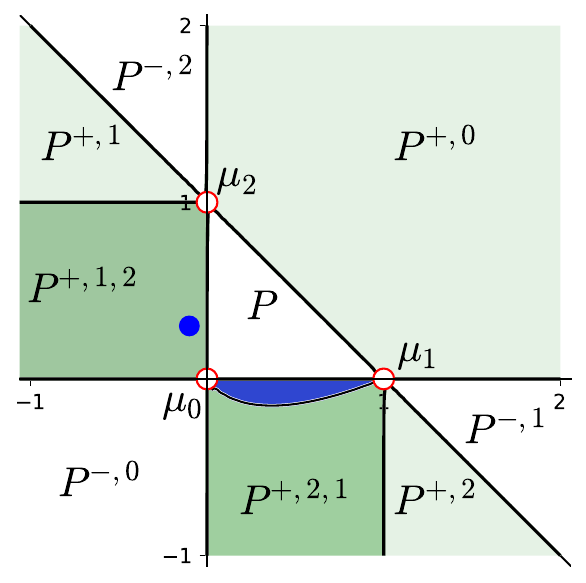}
\caption{{\small An illustration of the chambers as defined in  \eqref{Eq::ChamberNeg},\eqref{Eq::ChamberPos},\eqref{Eq::ChamberPosSub} for $P = \Conv( (0,0),(1,0),(0,1))$. The red circles denote positive exponent vectors $\cA_+ = \left\{ \begin{bmatrix} 0 \\ 0 \end{bmatrix},\begin{bmatrix} 1 \\ 0 \end{bmatrix}, \begin{bmatrix} 0 \\ 1 \end{bmatrix} \right\}$, the blue dot denote a negative exponent vector $\begin{bmatrix} -0.1 \\ 0.3  \end{bmatrix}$. The blue region contains all negative exponent vectors $\begin{bmatrix} x_2 \\  y_2  \end{bmatrix} \in P^{+,2,1}$ such that $\bar{\xi}_{B,\eps}$ has two critical points. }}\label{FIG4}
\end{figure}

\paragraph{\textbf{Acknowledgements. }} 
WD and JMR were partially supported by NSF grant CCF-1900881. 
MLT thanks Elisenda Feliu for useful discussions and comments on the 
manuscript. MLT was supported by the European Union under the Grant Agreement 
no.\ 101044561, POSALG. Views and opinions expressed are those of the 
author(s) only and do not necessarily reflect those of the European Union or 
European Research Council (ERC). Neither the European Union nor ERC can be 
held responsible for them. 

{\small

}

\end{document}